\documentclass[a4paper]{article}

\usepackage{amsfonts, amsmath, amssymb, amsthm}
\usepackage{enumitem}
\usepackage[T1]{fontenc}
\usepackage{Baskervaldx}
\usepackage{tocloft}
\usepackage{color, xcolor}
	\definecolor{Blue}{HTML}{3d25b9}
	\definecolor{Red}{HTML}{c00054} 
\usepackage[UKenglish]{babel}
\usepackage{csquotes}
\usepackage{booktabs}
\usepackage{tabularx}
\usepackage{microtype}

\usepackage[margin=25mm]{geometry}
\usepackage{sectsty}
	\sectionfont{\large}
	\subsectionfont{\normalsize}
	\subsubsectionfont{\normalsize}
\usepackage{titlesec}
	\titlespacing{\section}{0pt}{12pt}{0pt}
	\titlespacing{\subsection}{0pt}{6pt}{0pt}
\usepackage[hang,flushmargin]{footmisc}

\setlist{topsep=0pt,itemsep=0pt}

\newcommand{\dd}{\mathrm{d}}
\newcommand{\bb}{\mathbf{b}}
\newcommand{\h}{\hbar}
\newcommand{\MM}{\overline{\mathcal{M}}}
\newcommand{\N}{\overline{N}}
\newcommand{\Res}{\mathop{\mathrm{Res}}}
\newcommand{\zz}{\mathbf{z}}
\DeclareMathOperator{\sgn}{sgn}

\usepackage[colorlinks=true,linkcolor=Blue,citecolor=Blue]{hyperref}
	\hypersetup{breaklinks=true}
\usepackage[capitalise,noabbrev]{cleveref}
	\crefname{equation}{equation}{equations}

\theoremstyle{plain}
	\newtheorem{theorem}{Theorem}
	\newtheorem{proposition}[theorem]{Proposition}
	\newtheorem{corollary}[theorem]{Corollary}
	\newtheorem{lemma}[theorem]{Lemma}
	\newtheorem{conjecture}[theorem]{Conjecture}
	\numberwithin{theorem}{section}
\theoremstyle{definition}
	\newtheorem{definition}[theorem]{Definition}
	\newtheorem{example}[theorem]{Example}
\theoremstyle{remark}
	\newtheorem*{remark*}{Remark}

\title{Local topological recursion governs the enumeration of lattice points in $\overline{\mathcal M}_{g,n}$}
\author{Anupam Chaudhuri \and Norman Do \and Ellena Moskovsky}
\date{\today}

\begin{document}

\textbf{{\large Local topological recursion governs the enumeration of lattice points in $\overline{\mathcal M}_{g,n}$}}

\textbf{Anupam Chaudhuri, Norman Do and Ellena Moskovsky}

School of Mathematics, Monash University, VIC 3800 Australia \\
Email: \href{mailto:anupam.chaudhuri@monash.edu}{anupam.chaudhuri@monash.edu}, \href{mailto:norm.do@monash.edu}{norm.do@monash.edu}, \href{mailto:ellena.moskovsky@monash.edu}{ellena.moskovsky@monash.edu}

{\em Abstract.} The second author and Norbury initiated the enumeration of lattice points in the Deligne--Mumford compactifications of moduli spaces of curves. They showed that the enumeration may be expressed in terms of polynomials, whose top and bottom degree coefficients store psi-class intersection numbers and orbifold Euler characteristics of $\overline{\mathcal M}_{g,n}$, respectively. Furthermore, they ask whether the enumeration is governed by the topological recursion and whether the intermediate coefficients also store algebro-geometric information. In this paper, we prove that the enumeration does indeed satisfy the topological recursion, although with a modification to the initial spectral curve data. Thus, one can consider this to be one of the first known instances of a natural enumerative problem governed by the so-called {\em local topological recursion}. Combining the present work with the known relation between local topological recursion and cohomological field theory should uncover the geometric meaning of the intermediate coefficients of the aforementioned polynomials.

\emph{Acknowledgements.} The second author was supported by the Australian Research Council grant DP180103891. The authors would like to thank Danilo Lewanski for interesting discussions and Paul Norbury for valuable feedback on an early version of the paper.

\emph{2010 Mathematics Subject Classification.} 14N10 (primary), 05A15, 14N35, 30F30.

~

\hrule

~

\tableofcontents

~

\hrule

~

\section{Introduction} \label{sec:introduction}

Norbury proved that a certain count of lattice points in the moduli space of curves ${\mathcal M}_{g,n}$ stores information about its intersection theory and orbifold Euler characteristic~\cite{nor10}. He furthermore showed that the enumeration is governed by the topological recursion of Chekhov, Eynard and Orantin~\cite{che-eyn06,eyn-ora07,nor13}. More recently, Andersen, Chekhov, Norbury and Penner use the general theory that identifies topological recursion with the Givental formalism to relate this enumeration to cohomological field theory~\cite{and-che-nor-pen15a,and-che-nor-pen15b,dun-ora-sha-spi14}.

The second author and Norbury introduced the related count of lattice points in $\overline{\mathcal M}_{g,n}$, the Deligne--Mumford compactification of the moduli space of curves~\cite{do-nor11}. For positive integers $b_1, b_2, \ldots, b_n$, they define
\[
\overline{\mathcal Z}_{g,n}(b_1, b_2, \ldots, b_n) \subset \overline{\mathcal M}_{g,n}
\]
to be the set of genus $g$ stable curves $\Sigma$ with marked points $(p_1, p_2, \ldots, p_n)$ such that there exists a morphism $f: \Sigma \to \mathbb{CP}^{1}$ satisfying the following conditions.
\begin{enumerate}[label=(C\arabic*)]
\item The morphism $f$ has degree $b_1 + b_2 + \cdots + b_n$ and is regular over $\mathbb{P}^1 \setminus \{0, 1, \infty\}$.
\item The ramification profile over $1 \in \mathbb{CP}^1$ is of the form $(2, 2, \ldots, 2)$ and the ramification profile over $\infty \in \mathbb{CP}^1$ is of the form $(b_1, b_2, \ldots, b_n)$, with ramification order $b_k$ occurring at the point $p_k \in \Sigma$.
\item Each point over $0 \in \mathbb{CP}^{1}$ has ramification order at least 2 or is a node of $\Sigma$.
\end{enumerate}

The set $\overline{\mathcal Z}_{g,n}(b_1, b_2, \ldots, b_n)$ is typically the union of a finite set of discrete points in $\overline{\mathcal M}_{g,n}$ with higher-dimensional components that are naturally products of moduli spaces of curves. The latter arise from maps $f: \Sigma \to \mathbb{CP}^{1}$ that have {\em ghost components} --- that is, irreducible components of $\Sigma$ that map entirely to $0 \in \mathbb{CP}^1$. To properly ``count'' points in $\overline{\mathcal Z}_{g,n}(b_1, b_2, \ldots, b_n)$, one needs to account for both the orbifold nature of $\overline{\mathcal M}_{g,n}$ and the existence of these ghost components. This can be conveniently expressed via the orbifold Euler characteristic as follows.

\begin{definition} \label{def:Ngn}
For positive integers $b_1, b_2, \ldots, b_n$, define 
\[
\N_{g,n}(b_1, b_2, \ldots, b_n) = \chi \left(\overline{\mathcal Z}_{g,n}(b_1, b_2, \ldots, b_n) \right).
\]
\end{definition}

The enumeration $\N_{g,n}$ enjoys the following properties, which can be found in the existing literature~\cite{do-nor11} and are explained in greater detail in~\cref{subsec:Ngn}. 
\begin{itemize}
\item {\em Quasi-polynomiality.} For $(g,n) \neq (0,1)$ or $(0,2)$, $\N_{g,n}(b_1, b_2, \ldots, b_n)$ is a symmetric quasi-polynomial in $b_1^2, b_2^2, \ldots, b_n^2$ of degree $\dim_\mathbb{C} \overline{\mathcal M}_{g,n} = 3g-3+n$. We use the term {\em quasi-polynomial} to refer to a function on $\mathbb{Z}_+^n$ that is polynomial on each fixed parity class. Observe that this allows us to extend $\N_{g,n}(b_1, b_2, \ldots, b_n)$ to evaluation at $b_i = 0$.
\item {\em Combinatorial recursion.} The enumeration $\N_{g,n}(b_1, b_2, \ldots, b_n)$ can be interpreted as a weighted count of combinatorial objects known as {\em stable ribbon graphs}. From this interpretation, one can deduce an effective recursion to calculate $\N_{g,n}(b_1, b_2, \ldots, b_n)$.
\item {\em Psi-class intersection numbers.} The top degree coefficients of the quasi-polynomial $\N_{g,n}$ store psi-class intersection numbers on $\overline{\mathcal M}_{g,n}$.
\item {\em Orbifold Euler characteristics.} The quasi-polynomial $\N_{g,n}$ satisfies $\N_{g,n}(0, 0, \ldots, 0) = \chi(\overline{\mathcal M}_{g,n})$.
\end{itemize}

We previously mentioned that the enumeration of lattice points in the uncompactified moduli space of curves ${\mathcal M}_{g,n}$ is governed by the topological recursion and consequently, related to cohomological field theory. It is certainly natural to seek analogous results in the context of the compactified enumeration $\N_{g,n}$. In this regard, the second author and Norbury originally state the following.
\begin{enumerate}[label=(\alph*)]
\item {\itshape ``It would be interesting to know whether the compactified lattice point polynomials can be used to define multidifferentials which also satisfy a topological recursion.''} \cite[p.~2343]{do-nor11}
\item {\itshape ``We remark that it is currently unknown whether or not the intermediate coefficients of $\N_{g,n}(\bb)$ store topological information about $\MM_{g,n}$.''} \cite[p.~2323]{do-nor11}
\end{enumerate}

In this paper, we settle problem (a) above by proving that the enumeration $\N_{g,n}$ is indeed governed by the topological recursion, although with a modification to the initial spectral curve data that is explained below. Although problem (b) above remains unresolved, our main theorem should allow one to invoke the general theory that identifies topological recursion with the Givental formalism to yield a connection to cohomological field theory~\cite{dun-ora-sha-spi14}. This would then provide a relation between the intermediate coefficients of $\N_{g,n}(b_1, b_2, \ldots, b_n)$ and the intersection theory of $\MM_{g,n}$. We aim to report on this work in the future.

The main result of the present work is the following.

\begin{theorem} \label{thm:TR}
Topological recursion applied to the {\bf local spectral curve} $\mathbb{C}^*$ equipped with the data
\begin{equation} \label{eq:spectralcurve}
x(z) = z + \frac{1}{z}, \qquad y(z) = z \qquad \text{and} \qquad \omega_{0,2}(z_1, z_2) = \frac{\dd z_1 \otimes \dd z_2}{(z_1 - z_2)^2} + \frac{\dd z_1 \otimes \dd z_2}{z_1 z_2}
\end{equation}
produces multidifferentials whose expansions at $z_i = 0$ satisfy
\[
\omega_{g,n}(z_1, z_2, \ldots, z_n) = \sum_{b_1, b_2, \ldots, b_n = 0}^{\infty} \N_{g,n}(b_1, b_2, \ldots, b_n) \, \prod_{i=1}^n [b_i] z_i^{b_i - 1} \, \dd z_i, \qquad \text{for $(g,n) \neq (0,1)$ or $(0,2)$.}
\]
Here, we use the notation $[b] = b$ for $b$ positive and $[0] = 1$. 
\end{theorem}

The most notable aspect of the theorem is the nature of the spectral curve involved, which can be considered local rather than global, in the sense discussed below.

\begin{itemize}
\item {\em Global topological recursion.}\footnote{We use the expression {\em global topological recursion} to contrast it with its local counterpart. However, this is not to be confused with the global version of topological recursion introduced by Bouchard and Eynard~\cite{bou-eyn13}.} In the foundational literature on topological recursion, a spectral curve is defined to be the data $({\mathcal C}, x, y, T)$, where ${\mathcal C}$ is a compact Riemann surface, $x$ and $y$ are meromorphic functions on ${\mathcal C}$, and $T$ is a Torelli marking on ${\mathcal C}$ --- that is, a choice of symplectic basis for $H_1({\mathcal C}; \mathbb{Z})$~\cite{eyn-ora07,eyn-ora09}. One usually also imposes some mild regularity conditions on this data, although they play no role in the present discussion. The global topological recursion then recursively produces so-called correlation differentials $\omega_{g,n}$ for integers $g \geq 0$ and $n \geq 1$. In particular, $\omega_{0,2}(z_1, z_2)$ is defined implicitly by the fact that it has double poles without residue along the diagonal $z_1=z_2$, is holomorphic away from the diagonal, and is normalised on the $\mathcal{A}$-cycles of the Torelli marking via the equation
\[
\oint_{\mathcal{A}_i} \omega_{0,2}(z_1, z_2) = 0, \qquad \text{ for } i = 1, 2, \ldots, \text{genus}({\mathcal C}).
\]
The compact nature of $\mathcal C$ ensures that $\omega_{0,2}$ is uniquely defined from the spectral curve data. A consequence of the global topological recursion is that for $(g,n) \neq (0,1)$ or $(0,2)$, the correlation differentials $\omega_{g,n}$ have poles only at the branch points of the spectral curve, where $\dd x$ vanishes.

\item {\em Local topological recursion.} One can observe that the global topological recursion actually only requires the local information of the meromorphic functions $x, y$ and the bidifferential $\omega_{0,2}$ at the branch points of the spectral curve, in order to produce the correlation differentials. Thus, one can more generally define topological recursion on spectral curves comprising isolated local germs of $x, y$ and $\omega_{0,2}$, without requiring the existence of a global compact Riemann surface on which this data can be defined. In particular, the local topological recursion requires $\omega_{0,2}$ to become part of the spectral curve data. This viewpoint was promoted by Dunin-Barkowksi, Orantin, Shadrin and Spitz in their work relating topological recursion to Givental's approach to cohomological field theory~\cite{dun-ora-sha-spi14}.
\end{itemize}

The spectral curve of~\cref{thm:TR} is local in the sense that the data cannot be extended to $\mathbb{CP}^1$ such that $\omega_{0,2}$ satisfies the conditions of the global topological recursion. The simple poles of $\omega_{0,2}$ at $z_1=0$ and $z_2=0$ lead to the correlation differentials $\omega_{g,n}(z_1, z_2, \ldots, z_n)$ having simple poles at $z_i=0$ more generally. This departs from the usual behaviour exhibited by the global topological recursion, in which the poles appear only at the branch points of the spectral curve, which correspond to $z_i = \pm 1$ in our case.

It was previously unclear whether there were benefits to using the local version of the topological recursion beyond the more general viewpoint it afforded. Indeed, Dunin-Barkowski~\cite{dun18} states that {\em ``local topological recursion (to the moment) lacks interesting applications or profound meaning separate from what originates from ordinary (global) spectral curve topological recursion''}. \cref{thm:TR} above now provides an instance of the local topological recursion applied to a natural enumerative problem. The only other such example in the literature known to us is the concurrent work of Andersen et al., which relates Masur--Veech volumes to local topological recursion on the Airy spectral curve, equipped with an interesting choice of $\omega_{0,2}$~\cite{and-bor-cha-del-gia-lew-whe19}.

The proof of~\cref{thm:TR} adopts a general strategy that has been previously employed to show that topological recursion governs enumerative problems, such as counting lattice points in ${\mathcal M}_{g,n}$~\cite{nor13} and several variants of Hurwitz numbers~\cite{bou-her-liu-mul14,do-dye-mat17,do-lei-nor16,eyn-mul-saf11}. Minor technical difficulties arise from the modification to $\omega_{0,2}$, which introduces logarithmic terms into the topological recursion kernel.

The general theory of topological recursion allows one to calculate so-called symplectic invariants $F_g \in \mathbb{C}$ for $g = 0, 1, 2, \ldots$ and to deduce relations known as string and dilaton equations. Thus, we have the following immediate consequence of our main result, which previously appeared in the literature with an alternative proof~\cite{do-nor11}.

\begin{corollary} \label{cor:stringdilaton} ~
The string and dilaton equations for the topological recursion imply the following known relations, respectively, for $(g,n) \neq (0,1)$ or $(0,2)$ and $b_1, b_2, \ldots, b_n \geq 0$. The hat over $b_k$ in the first equation denotes the fact that we omit it as an argument.
\[
\N_{g,n+1}(1, b_1, b_2, \ldots, b_n) = \sum_{k=1}^{n} \sum_{a=0}^{b_k-1} [a] \, \N_{g,n}(a, b_1, \ldots, \widehat{b}_k, \ldots, b_n)
\]
\[
\N_{g,n+1}(2, b_1, b_2, \ldots, b_n) - \N_{g,n+1}(0, b_1, b_2, \ldots, b_n) = (2g - 2 + n) \, \N_{g,n}(b_1, b_2, \ldots, b_n)
\]
\end{corollary}

One potential application of the present work is to give an explicit relation between the enumeration $\N_{g,n}$ and the algebraic geometry of $\overline{\mathcal M}_{g,n}$. A priori, one might expect such a relation due to the definition of $\N_{g,n}(b_1, b_2, \ldots, b_n)$ as a virtual count of the set $\overline{\mathcal Z}_{g,n}(b_1, b_2, \ldots, b_n) \subset \overline{\mathcal M}_{g,n}$. Furthermore, we note that $\overline{\mathcal Z}_{g,n}(b_1, b_2, \ldots, b_n)$ may alternatively be interpreted as a subset of $\overline{\mathcal M}_{g,n}(\mathbb{CP}^1; \sum b_i)$, the moduli space of stable maps, making a connection with the Gromov--Witten theory of the sphere. \cref{thm:TR} now provides a promising pathway towards the ultimate proof of a relation between the enumeration $\N_{g,n}$ and the intersection theory of $\overline{\mathcal M}_{g,n}$ via the identification of topological recursion with Givental's formula~\cite{dun-ora-sha-spi14}.

A glance at~\cref{app:data} naturally leads to the conjecture below. The data supports the speculation that the coefficients of $\N_{g,n}$ store algebro-geometric content.

\begin{conjecture} \label{con:positivity}
The polynomials underlying the quasi-polynomial $\N_{g,n}$ have non-negative coefficients.
\end{conjecture}

The structure of the paper is as follows.
\begin{itemize}
\item In Section 2, we briefly review the relevant definitions and results concerning the count of lattice points in $\overline{\mathcal M}_{g,n}$ appearing in the previous work of the second author and Norbury~\cite{do-nor11}. In particular, we aim to reveal the sense in which $\N_{g,n}(b_1, b_2, \ldots, b_n)$ counts lattice points in $\overline{\mathcal M}_{g,n}$ and explain why the enumeration is natural. We furthermore provide a definition and concise exposition of the topological recursion, for both the global and local settings.

\item In Section 3, we give the proof of~\cref{thm:TR}. This requires some preliminary work to express the combinatorial recursion for $\N_{g,n}$ in terms of natural generating functions and to deduce some analytic structure for them. Finally, we match the combinatorial recursion with the topological recursion, adopting a general strategy that has previously been employed in the literature to prove that certain enumerative problems satisfy the topological recursion.

\item In Section 4, we discuss some further issues stemming from our work. We prove the consequences stated above as~\cref{cor:stringdilaton} and mention the potential relation to cohomological field theory. We conclude with some brief remarks on quantum curves and the question of where the local spectral curve came from. An answer to this may help to identify other interesting instances and applications of the local topological recursion.

\item In Appendix A, we present a table of the polynomials underlying the enumeration $\N_{g,n}$ for small values of $g$ and $n$.
\end{itemize}

\section{Background} \label{sec:background}

\subsection{Counting lattice points in \texorpdfstring{$\overline{\mathcal M}_{g,n}$}{M{g,n}}} \label{subsec:Ngn}

In this section, we discuss the enumeration given in \cref{def:Ngn}. Our exposition aims to reveal the sense in which $\N_{g,n}(b_1, b_2, \ldots, b_n)$ counts lattice points in $\overline{\mathcal M}_{g,n}$ and explain why the enumeration is natural. We begin by recalling a definition for a ribbon graph that is well-suited to our purposes.

\begin{definition}
A {\em ribbon graph} is a finite graph embedded into an oriented compact surface such that its complement is a disjoint union of topological disks, which we call {\em faces}. We say that a ribbon graph has {\em type $(g,n)$} if its underlying surface is connected with genus $g$ and there are $n$ faces labelled from 1 up to~$n$. {\em Unless otherwise stated, we exclusively consider ribbon graphs in which every vertex has degree at least two.}

An {\em isomorphism} between two ribbon graphs comprises bijections between their respective vertices, oriented edges and faces that are realised by an orientation-preserving homeomorphism between their underlying surfaces and preserve all adjacencies and face labels.
\end{definition}

A construction of Strebel associates to a smooth genus $g$ curve with $n$ marked points, each decorated by a positive real number, a metric ribbon graph of type $(g,n)$~\cite{str}. A {\em metric ribbon graph} is a ribbon graph whose vertex degrees are at least three, with a positive real number associated to each edge. The metric structure naturally endows the faces with perimeters that are equal to the original decorations on the marked points. The construction allows one to form the orbifold cell decompositions
\[
{\mathcal M}_{g,n} \times \mathbb{R}_+^n \cong \bigsqcup_\Gamma {\mathcal P}_\Gamma \qquad \text{and} \qquad {\mathcal M}_{g,n} \cong \bigsqcup_{\Gamma} {\mathcal P}_\Gamma(b_1, b_2, \ldots, b_n),
\]
where the unions are over the finite set of ribbon graphs of type $(g,n)$ whose vertex degrees are at least three. The cell ${\mathcal P}_\Gamma$ consists of the metric ribbon graphs whose underlying ribbon graph is equal to $\Gamma$. The latter decomposition is obtained from the former by prescribing the face perimeters $b_1, b_2, \ldots, b_n \in \mathbb{R}_+$.

These cell decompositions are fundamental in Kontsevich's proof of Witten's conjecture, which proceeds by calculation of the volume of the moduli space with respect to a particular symplectic structure~\cite{kon92}. This is closely related to Mirzakhani's calculation of the Weil--Petersson volumes of moduli spaces of hyperbolic surfaces and indeed, arises as a particular limit of it~\cite{do10, mir07}. Norbury proposed to discretise the volume calculation, by restricting to positive integer values of $b_1, b_2, \ldots, b_n$ and counting lattice points in the resulting integral polytopes ${\mathcal P}_\Gamma(b_1, b_2, \ldots, b_n)$. These correspond to metric ribbon graphs with vertex degrees at least three and integral edge lengths or equivalently, ribbon graphs with vertex degrees at least two. Thus, we have the notion of lattice points in ${\mathcal M}_{g,n}$ and the associated enumeration possesses a variety of interesting properties~\cite{and-che-nor-pen15a, and-che-nor-pen15b, nor10, nor13}.

Ribbon graphs appear in diverse mathematical settings and lie in natural bijection with several other combinatorial and geometric objects~\cite{lan-zvo04}. For our purposes, it is advantageous to associate to a ribbon graph of type $(g,n)$ with perimeters $b_1, b_2, \ldots, b_n$ a branched cover $f: \Sigma \to \mathbb{CP}^1$ from a genus $g$ smooth curve with $n$ marked points $(p_1, p_2, \ldots, p_n)$ satisfying the following conditions.
\begin{enumerate}[label=(C\arabic*)]
\item The morphism $f$ has degree $b_1 + b_2 + \cdots + b_n$ and is regular over $\mathbb{P}^1 \setminus \{0, 1, \infty\}$.
\item The ramification profile over $1 \in \mathbb{CP}^1$ is of the form $(2, 2, \ldots, 2)$ and the ramification profile over $\infty \in \mathbb{CP}^1$ is of the form $(b_1, b_2, \ldots, b_n)$, with ramification order $b_k$ occurring at the point $p_k \in \Sigma$.
\item[($\widehat{\mathrm C}3$)] Each point over $0 \in \mathbb{CP}^{1}$ has ramification order at least 2.
\end{enumerate}
The construction is such that the ribbon graph is recovered by taking $f^{-1}([0,1]) \subset \Sigma$. The preimages of $0, 1, \infty$ then naturally correspond to vertices, edges and faces of the ribbon graph, respectively. The labelling of the marked points gives rise to a labelling of the faces of the ribbon graph. It is common to describe a ribbon graph via the permutation model~\cite{lan-zvo04}, and the branched cover encodes the associated triple of permutations via its monodromy around 0, 1 and $\infty$.

This perspective allows one to interpret the enumeration of ribbon graphs as an enumeration of maps to~$\mathbb{CP}^1$, thereby making thematic contact with Gromov--Witten theory. However, we observe that stable curves and stable maps do not feature in this count. The previous work of the second author and Norbury aims to ``compactify'' the count, with the aim of introducing an enumeration that is more natural from the Gromov--Witten viewpoint~\cite{do-nor11}. The basic premise is to consider morphisms $f: \Sigma \to \mathbb{CP}^1$ with domains that are stable curves and to consequently adjust condition ($\widehat{\mathrm C}3$) above to the following natural generalisation, as mentioned in~\cref{sec:introduction}.
\begin{enumerate}
\item[(C3)] Each point over $0 \in \mathbb{CP}^{1}$ has ramification order at least 2 or is a node of $\Sigma$.
\end{enumerate}

Again, one can consider the preimage of the interval $[0,1]$ under a morphism satisfying conditions (C1), (C2) and (C3). The domain curve may be nodal and each component that maps with positive degree to $\mathbb{CP}^1$ receives the structure of a ribbon graph. On the other hand, there may be components that map with degree zero to $0 \in \mathbb{CP}^1$ and we refer to these as {\em ghost components}. The structure of such a preimage is encapsulated by the combinatorial notion of a {\em stable ribbon graph}, which is explained in the previous work of the second author and Norbury~\cite{do-nor11}. Forgoing the details, one can interpret $\N_{g,n}(b_1, b_2, \ldots, b_n)$ as a weighted enumeration of stable ribbon graphs of type $(g,n)$ with face perimeters prescribed by $b_1, b_2, \ldots, b_n$. The weight of a stable ribbon graph is the reciprocal of its number of automorphisms multiplied by a contribution of $\chi(\overline{\mathcal M}_{g',n'})$ for each maximal ghost component of genus $g'$ that is connected to non-ghost components by $n'$ nodes. The notion of stable ribbon graph is the primary tool used to prove the combinatorial recursion for $\N_{g,n}(b_1, b_2, \ldots, b_n)$ stated below.

We record the following properties of the enumeration $\N_{g,n}(b_1, b_2, \ldots, b_n)$, the first two of which are crucial in the proof of our main result in~\cref{sec:proof}.

\begin{theorem}[Do and Norbury~\cite{do-nor11}] \label{thm:propNgn} ~
\begin{enumerate}
\item {\em Quasi-polynomiality.} For $(g,n) \neq (0,1)$ or $(0,2)$, $\N_{g,n}(b_1, b_2, \ldots, b_n)$ is a symmetric quasi-polynomial in $b_1^2, b_2^2, \ldots, b_n^2$ of degree $\dim_\mathbb{C} \overline{\mathcal M}_{g,n} = 3g-3+n$. We use the term {\em quasi-polynomial} to refer to a function on $\mathbb{Z}_+^n$ that is polynomial on each fixed parity class. Observe that $\N_{g,n}(b_1, b_2, \ldots, b_n) = 0$ whenever $b_1 + b_2 + \cdots + b_n$ is odd and that quasi-polynomiality allows us to extend $\N_{g,n}(b_1, b_2, \ldots, b_n)$ to evaluation at $b_i = 0$.

\item {\em Combinatorial recursion.} For $2g-2+n \geq 2$ and $b_1, b_2, \ldots, b_n \geq 0$, we have the following equation, where $S = \{1, 2, \ldots, n\}$ and $\bb_I = (b_{i_1}, b_{i_2}, \ldots, b_{i_k})$ for $I = \{i_1, i_2, \ldots, i_k\}$.
\begin{align*}
&\left(\sum_{i=1}^{n} b_i\right) \N_{g,n}(\bb _S) = \sum_{i < j} \sum_{\substack{p+q= b_i+b_j \\ q \text{ even}}} [p]q \, \N_{g,n-1}(p, \bb_{S\setminus \{i,j\}}) \nonumber \\
& \qquad + \frac{1}{2} \sum_i \sum_{\substack{p+q+r=b_i \\ r \text{ even}}} [p] [q] r \Bigg[ \N_{g-1,n+1}(p,q, \bb_{S \setminus \{i\}}) + \sum_{\substack{g_1+g_2=g \\ I \sqcup J = S \setminus \{i\}}}^{\mathrm{stable}} \N_{g_1,|I|+1}(p,\bb_I) \, \N_{g_2,|J|+1}(q,\bb_J) \Bigg]
\end{align*}
In the summations, $p, q, r$ vary over all non-negative integers and we use the notation $[p] = p$ for $p$ positive and $[0] = 1$. The word $\mathrm{stable}$ over the final summation denotes that we exclude all terms with $\N_{0,1}$ or $\N_{0,2}$.

\item {\em Psi-class intersection numbers.} For non-negative integers $\alpha_1 + \alpha_2 + \cdots + \alpha_n = 3g - 3 + n$, the coefficient of $b_1^{2\alpha_1}b_2^{2\alpha_2} \cdots b_n^{2\alpha_n}$ in any non-zero polynomial underlying $\N_{g,n}(b_1, b_2, \ldots, b_n)$ is equal to the psi-class intersection number
\[
\frac{1}{2^{5g-6+2n} \, \alpha_1! \, \alpha_2! \, \cdots \, \alpha_n!} \int_{\overline{\mathcal{M}}_{g,n}} \psi_1^{\alpha_1} \psi_2^{\alpha_2} \cdots \psi_n^{\alpha_n}.
\]

\item {\em Orbifold Euler characteristics.} Let $\chi_{g,n}$ denote the orbifold Euler characteristic of $\overline{\mathcal M}_{g,n}$. The quasi-polynomial $\N_{g,n}(b_1, b_2, \ldots, b_n)$ satisfies $\N_{g,n}(0,0,\ldots, 0) = \chi_{g,n}$. Furthermore, these numbers obey the following relation for $2g-2+n \geq 1$, with the convention $\chi_{0,1} = 0$ and $\chi_{0,2} = 1$.
\[
\chi_{g,n+1} = (2 - 2g - n) \, \chi_{g,n} + \frac{1}{2} \, \chi_{g-1,n+2} + \frac{1}{2} \sum_{\substack{g_1+g_2=g \\ i+j=n}} \binom{n}{i} \chi_{g_1, i+1} \, \chi_{g_2,j+1}
\] 
\end{enumerate}
\end{theorem}

\subsection{Topological recursion} \label{sec:TR}

The topological recursion of Chekhov, Eynard and Orantin emerged as an abstract formulation of loop equations from the theory of matrix models~\cite{che-eyn06, eyn-ora07}. Beyond the original applications to matrix models, topological recursion has been shown or conjectured to govern a diverse array of problems in mathematics and physics, including the following: intersection theory on moduli spaces of curves~\cite{eyn-ora07}; enumeration of ribbon graphs and hypermaps~\cite{do-man14, dum-mul-saf-sor13, dun-ora-pop-sha14, nor13}; various kinds of Hurwitz numbers~\cite{bou-her-liu-mul14,do-dye-mat17,do-kar16,do-lei-nor16,eyn-mul-saf11,dun-kra-pop-sha19}; Gromov--Witten invariants of $\mathbb{CP}^1$~\cite{dun-ora-sha-spi14,nor-sco14}; Gromov--Witten invariants of toric Calabi--Yau threefolds~\cite{bou-kle-mar-pas09,eyn-ora15,fan-liu-zon16}; and asymptotics of coloured Jones polynomials of knots~\cite{bor-eyn15,dij-fuj-man11}.

In general, the topological recursion takes as input a spectral curve and produces multidifferentials $\omega_{g,n}$ for integers $g \geq 0$ and $n \geq 1$, which we refer to as {\em correlation differentials}. If the underlying Riemann surface of the spectral curve is $\mathcal C$, then $\omega_{g,n}$ is a symmetric meromorphic section of the line bundle $\pi_1^*(T^*\mathcal{C}) \otimes \pi_2^*(T^*\mathcal{C}) \otimes \cdots \otimes \pi_n^*(T^*\mathcal{C})$ on the Cartesian product $\mathcal{C}^n$, where $\pi_i\colon \mathcal{C}^n \to \mathcal{C}$ denotes projection onto the $i$th factor. An explicit definition of topological recursion follows.

\begin{itemize}
\item {\bf Initial data.} A {\em global spectral curve} is a tuple $(\mathcal{C}, x, y, T)$, where ${\mathcal C}$ is a compact Riemann surface, $x$ and~$y$ are meromorphic functions on ${\mathcal C}$, and $T$ is a Torelli marking on ${\mathcal C}$ --- that is, a choice of symplectic basis for $H_1({\mathcal C}; \mathbb{Z})$. We furthermore require the zeroes of $\dd x$ to be simple and disjoint from the zeroes of $\dd y$.

\item {\bf Base cases.} Let $\omega_{0,1}(z_1) = -y(z_1) \, \dd x(z_1)$. Let $\omega_{0,2}(z_1, z_2)$ be the unique meromorphic bidifferential on $\mathcal C$ that has double poles without residue along the diagonal $z_1=z_2$, is holomorphic away from the diagonal, and is normalised on the $\mathcal{A}$-cycles of the Torelli marking via the equation
\[ 
\oint_{\mathcal{A}_i} \omega_{0,2}(z_1, z_2) = 0, \qquad \text{ for } i = 1, 2, \ldots, \text{genus}({\mathcal C}).
\]

\item {\bf Recursion.} For $2g-2+n>0$, the multidifferentials $\omega_{g,n}(z_1, z_2, \ldots, z_n)$ are defined recursively by the following equation, where $S = \{2, 3, \ldots, n\}$ and $\zz_I = (z_{i_1}, z_{i_2}, \ldots, z_{i_k})$ for $I = \{i_1, i_2, \ldots, i_k\}$.
\[
\omega_{g,n}(z_1, \zz_S) = \sum_{\alpha} \Res_{z=\alpha} K(z_1, z) \Bigg[ \omega_{g-1, n+1}\left(z, s(z), \zz _S \right) + \sum_{\substack{g_1+g_2=g \\ I\sqcup J=S}}^{\circ} \omega_{g_1, |I|+1}(z, \zz_I) \, \omega_{g_2, |J|+1}(s(z), \zz_J) \Bigg]
\]
The outer summation is over the zeroes $\alpha$ of $\dd x$, while the symbol $\circ$ over the inner summation denotes that we exclude all terms with $\omega_{0,1}$. The function $s(z)$ is the unique non-identity holomorphic map defined in a neighbourhood of the simple branch point $\alpha \in {\mathcal C}$ satisfying $x(s(z)) = x(z)$. Finally, the kernel $K(z_1, z)$ is defined by
\[
K(z_1, z) = - \frac{\int_{o}^{z} \omega_{0,2}(z_1, \cdot\,)}{\left[ y(z) - y(s(z)) \right] \, \dd x(z)},
\]
where $o$ can be taken to be an arbitrary point on the spectral curve.
\end{itemize}

As mentioned in the introduction, one can observe that the topological recursion as defined above actually only requires the local information of the meromorphic functions $x, y$ and the bidifferential $\omega_{0,2}$ at the branch points of the spectral curve, in order to produce the correlation differentials. Thus, one can more generally define topological recursion on spectral curves comprising isolated local germs of $x, y$ and $\omega_{0,2}$, without requiring the existence of a global compact Riemann surface on which this data can be defined. We refer to this generalisation as {\em local topological recursion}, which we may define as follows.

\begin{itemize}
\item {\bf Initial data.} A {\em local spectral curve} is a tuple $({\mathcal C}, x, y, \omega_{0,2})$, where ${\mathcal C}$ is a Riemann surface that may be non-compact and disconnected, $x$ and~$y$ are meromorphic functions on ${\mathcal C}$, and $\omega_{0,2}$ is a meromorphic bidifferential on $\mathcal C$ that has double poles without residue along the diagonal $z_1=z_2$ and is holomorphic away from the diagonal. We furthermore require the zeroes of $\dd x$ to be simple and disjoint from the zeroes of $\dd y$.

\item {\bf Base cases.} Let $\omega_{0,1}(z_1) = -y(z_1) \, \dd x(z_1)$ and let $\omega_{0,2}(z_1, z_2)$ be as in the initial data.
 
\item {\bf Recursion.} For $2g-2+n>0$, the multidifferentials $\omega_{g,n}(z_1, z_2, \ldots, z_n)$ are defined recursively by the same equation as in the global topological recursion above.
\end{itemize}

Local topological recursion was introduced by Dunin-Barkowksi, Orantin, Shadrin and Spitz in their work relating topological recursion to Givental's approach to cohomological field theory~\cite{dun-ora-sha-spi14}. There they take $\mathcal C = {\mathcal D}_1 \sqcup {\mathcal D}_2 \sqcup \cdots \sqcup {\mathcal D}_N$ to be a disjoint union of $N$ small disks and they endow ${\mathcal D}_i$ with a local coordinate $z$ such that $x(z) = z^2 + a_i$ in ${\mathcal D}_i$, for $i = 1, 2, \ldots, N$. The initial data can then be expressed in terms of the coefficients of the local expansions of $y(z)$ in ${\mathcal D}_i$ and $\omega_{0,2}(z_1, z_2)$ in ${\mathcal D}_i \times {\mathcal D}_j$ for $i, j = 1, 2, \ldots, N$.

The spectral curve of~\cref{thm:TR} is local in the sense that the data cannot be extended to a compact Riemann surface such that $\omega_{0,2}$ satisfies the conditions of the global topological recursion. Although we do not take this approach here, the spectral curve could alternatively have been presented more abstractly as the disjoint union of two small disks, corresponding to the two branch points.

\begin{example} \label{ex:omegabases}
Recall that the local spectral curve of~\cref{thm:TR} is $\mathbb{C}^*$ equipped with the data
\[
x(z) = z + \frac{1}{z}, \qquad y(z) = z \qquad \text{and} \qquad \omega_{0,2}(z_1, z_2) = \frac{\dd z_1 \otimes \dd z_2}{(z_1 - z_2)^2} + \frac{\dd z_1 \otimes \dd z_2}{z_1 z_2}.
\]
The branch points are the zeroes of $\dd x$ --- namely, $z = 1$ and $z = -1$. At both of these branch points, the local involution $s(z)$ is given by $s(z) = \frac{1}{z}$. Thus, the recursion kernel can be taken to be
\[
K(z_1, z) = - \frac{\int_o^{z} \omega_{0,2}(z_1, \cdot\,)}{\left[ y(z) - y(s(z)) \right] \, \dd x(z)} = - \frac{\int_\infty^{z} \frac{\dd z_1 \, \dd t}{(z_1-t)^2} + \int_1^z \frac{\dd z_1 \, \dd t}{z_1 \, t}}{\left[ y(z) - y(s(z)) \right] \, \dd x(z)} = - \left[ \frac{1}{z_1-z} + \frac{\log(z)}{z_1} \right] \frac{z^3}{(1-z^2)^2} \frac{\dd z_1}{\dd z}.
\]
The recursion produces the following formulas in the cases $(g,n) = (0,3)$ and $(1,1)$.
\begin{align*}
\frac{\omega_{0,3}(z_1, z_2, z_3)}{\dd z_1 \, \dd z_2 \, \dd z_3} &= \sum_{\alpha = \pm 1} \Res_{z=\alpha} \frac{K(z_1, z)}{{\dd z_1 \, \dd z_2 \, \dd z_3}} \left[ \omega_{0,2}(z, z_2) \, \omega_{0,2}(\tfrac{1}{z}, z_3) + \omega_{0,2}(z, z_3) \, \omega_{0,2}(\tfrac{1}{z}, z_2) \right] \\
&= \sum_{\alpha = \pm 1} \Res_{z=\alpha} \left[ \frac{1}{z_1-z} + \frac{\log(z)}{z_1} \right] \frac{z^3}{(1-z^2)^2} \left[ \frac{1}{(z-z_2)^2 \, (1-zz_3)^2} + \frac{1}{(z-z_3)^2 \, (1-zz_2)^2} \right] \dd z \\
&= \frac{1}{2 z_1 z_2 z_3} \bigg[ \prod_{i=1}^3 \frac{z_i^2 - z_i + 1}{(z_i-1)^2} + \prod_{i=1}^3 \frac{z_i^2 + z_i + 1}{(z_i+1)^2} \bigg] \\
\frac{\omega_{1,1}(z_1)}{\dd z_1} &= \sum_{\alpha = \pm 1} \Res_{z=\alpha} \frac{K(z_1, z)}{\dd z_1} \, \omega_{0,2}(z, \tfrac{1}{z}) \\
&= \sum_{\alpha = \pm 1} \Res_{z=\alpha} \left[ \frac{1}{z_1-z} + \frac{\log(z)}{z_1} \right] \frac{z^3}{(1-z^2)^2} \left( \frac{1}{(z^2 - 1)^2} + 1 \right) \, \dd z \\
&= \frac{5z_1^8 - 8z_1^6 + 18z_1^4 - 8z_1^2 + 5}{12 z_1 (z_1^2-1)^4}
\end{align*}
\end{example}

\section{Proof of the main theorem} \label{sec:proof}

For the proof of~\cref{thm:TR}, we adopt a general strategy that has been previously used to prove the topological recursion for enumerative problems, such as counting lattice points in uncompactified moduli spaces of curves~\cite{nor13} and various kinds of Hurwitz numbers~\cite{eyn-mul-saf11,bou-her-liu-mul14,do-lei-nor16,do-dye-mat17}. The modification to $\omega_{0,2}$ in our result adds minor technical difficulties, since logarithmic terms are introduced into the topological recursion kernel. We break down the proof into the following parts.

\begin{enumerate}
\item Define natural multidifferentials $\Omega_{g,n}(z_1, z_2, \ldots, z_n)$ for the enumerative problem and use the quasi-polynomiality of~\cref{thm:propNgn} to deduce analytic and symmetry properties for $\Omega_{g,n}(z_1, z_2, \ldots, z_n)$ (\cref{prop:vectorspace}).

\item Express the combinatorial recursion of \cref{thm:propNgn} in terms of the aforementioned multidifferentials $\Omega_{g,n}(z_1, z_2, \ldots, z_n)$ (\cref{prop:genfun}).
	
\item Break the natural symmetry of the recursion for $\Omega_{g,n}(z_1, z_2, \ldots, z_n)$ by taking the symmetric part with respect to $z_1$, using the symmetry properties of \cref{prop:vectorspace} (\cref{prop:asymmetricgf}).

\item Use the fact that a rational differential form is equal to the sum of its principal parts, where the principal part of $\Omega(z_1)$ at $z_1 = \alpha$ may be defined by
\begin{equation} \label{eq:principalpart}
\Res_{z=\alpha} \frac{\dd z_1}{z_1 - z} \Omega(z).
\end{equation}
Finally, match the resulting recursion for the multidifferentials $\Omega_{g,n}(z_1, z_2, \ldots, z_n)$ with the topological recursion for the correlation differentials $\omega_{g,n}(z_1, z_2, \ldots, z_n)$.
\end{enumerate}

These four steps are carried out in the following four subsections, respectively.

\subsection{Structure of the enumeration}

From the enumeration $\N_{g,n}(b_1, b_2, \ldots, b_n)$ of~\cref{def:Ngn}, we define the following formal multidifferentials.
\begin{equation} \label{eq:defOmega}
\Omega_{g,n}(z_1, z_2, \ldots, z_n) = \sum_{b_1, b_2, \ldots, b_n=0}^\infty \N_{g,n}(b_1, b_2, \ldots, b_n) \prod_{i=1}^n [b_i] z_i^{b_i-1} \, \dd z_i
\end{equation}
\cref{thm:TR} is essentially the statement that the correlation differentials produced by the topological recursion applied to the spectral curve of~\cref{eq:spectralcurve} satisfy
\[
\Omega_{g,n}(z_1, z_2, \ldots, z_n) = \omega_{g,n}(z_1, z_2, \ldots, z_n), \qquad \text{for } (g,n) \neq (0,1) \text{ or } (0,2).
\]
The primary aim is to understand the structure of $\Omega_{g,n}(z_1, z_2, \ldots, z_n)$, which will play a crucial role in the proof of~\cref{thm:TR}. The quasi-polynomiality of $\N_{g,n}(b_1, b_2, \ldots, b_n)$ stated in~\cref{thm:propNgn} is equivalent to the fact that for $(g,n) \neq (0,1)$ or $(0,2)$,
\begin{equation} \label{eq:vectorspace}
\Omega_{g,n}(z_1, z_2, \ldots, z_n) \in V(z_1) \otimes V(z_2) \otimes \cdots \otimes V(z_n),
\end{equation}
where we define the vector space $V(z)$ as follows.

\begin{definition}
Define the complex vector space of differential forms
\[
V(z) = \left\{ \sum_{b=0}^{\infty} \, [b] Q(b) z^{b-1} \, \dd z ~\Big|~ Q(b) \text{ is a quasi-polynomial in } b^2 \right\}.
\]
\end{definition}

\begin{lemma} \label{lem:basis}
The vector space $V(z)$ has the basis $\left\{\xi^{\mathrm{even}}_0(z), \xi^{\mathrm{odd}}_0(z), \xi^{\mathrm{even}}_1(z), \xi^{\mathrm{odd}}_1(z), \xi^{\mathrm{even}}_2(z), \xi^{\mathrm{odd}}_2(z), \ldots \right\}$, where
\[
\xi^{\mathrm{even}}_k(z) = \frac{\dd}{\dd z} \bigg( z \frac{\dd}{\dd z} \bigg)^{2k} \frac{z^2}{1-z^2} \, \dd z + \frac{\delta_{k,0}}{z} \, \dd z \qquad \text{and} \qquad \xi^{\mathrm{odd}}_k(z) = \frac{\dd}{\dd z} \bigg( z \frac{\dd}{\dd z} \bigg)^{2k} \frac{z}{1-z^2} \, \dd z.
\]
\end{lemma}

\begin{proof}
Begin by observing that a quasi-polynomial is a unique linear combination of monomials, acting on either even or odd arguments. So we have the following basis for $V(z)$, as $k$ varies over the non-negative integers.
\begin{align*}
\xi^{\mathrm{even}}_k(z) &= \sum_{\substack{b \geq 0\\ b \text{ even}}} \, [b] \cdot b^{2k} z^{b-1} \, \dd z & \xi^{\mathrm{odd}}_k(z) &= \sum_{\substack{b \geq 0 \\ b \text{ odd}}} \, [b] \cdot b^{2k} z^{b-1} \, \dd z \\
&= \frac{\dd}{\dd z} \bigg( z \frac{\dd}{\dd z} \bigg)^{2k} \sum_{\substack{b > 0 \\ b \text{ even}}} z^b \, \dd z + \frac{\delta_{k,0}}{z} \, \dd z &&= \frac{\dd}{\dd z} \bigg( z \frac{\dd}{\dd z} \bigg)^{2k} \sum_{\substack{b > 0 \\ b \text{ odd}}} z^b \, \dd z \\
&= \frac{\dd}{\dd z} \bigg( z \frac{\dd}{\dd z} \bigg)^{2k} \frac{z^2}{1-z^2} \, \dd z + \frac{\delta_{k,0}}{z} \, \dd z &&= \frac{\dd}{\dd z} \bigg( z \frac{\dd}{\dd z} \bigg)^{2k} \frac{z}{1-z^2} \, \dd z \qedhere
\end{align*}
\end{proof}

\begin{example} \label{ex:Omegabases}
It was previously shown that~\cite{do-nor11}
\[
\N_{0,3}(b_1, b_2, b_3) = \begin{cases}
1, & b_1+b_2+b_3 \text{ even}, \\
0, & b_1+b_2+b_3 \text{ odd},
\end{cases}
\qquad \text{and} \qquad
\N_{1,1}(b_1) = \begin{cases}
\frac{1}{48}(b_1^2 + 20), & b_1 \text{ even}, \\
0, & b_1 \text{ odd}.
\end{cases}
\]
Hence, we can express the corresponding generating differentials in terms of the basis elements of~\cref{lem:basis}.
\begin{align*}
\Omega_{0,3}(z_1, z_2, z_3) ={}& \zeta_0^{\mathrm{even}}(z_1) \, \zeta_0^{\mathrm{even}}(z_2) \, \zeta_0^{\mathrm{even}}(z_3) + \zeta_0^{\mathrm{even}}(z_1) \, \zeta_0^{\mathrm{odd}}(z_2) \, \zeta_0^{\mathrm{odd}}(z_3) \nonumber \\
&+ \zeta_0^{\mathrm{odd}}(z_1) \, \zeta_0^{\mathrm{even}}(z_2) \, \zeta_0^{\mathrm{odd}}(z_3) + \zeta_0^{\mathrm{odd}}(z_1) \, \zeta_0^{\mathrm{odd}}(z_2) \, \zeta_0^{\mathrm{even}}(z_3) \nonumber \\
={}& \frac{\dd z_1 \, \dd z_2 \, \dd z_3}{2 z_1 z_2 z_3} \bigg[ \prod_{i=1}^3 \frac{z_i^2 - z_i + 1}{(z_i-1)^2} + \prod_{i=1}^3 \frac{z_i^2 + z_i + 1}{(z_i+1)^2} \bigg] \\
\Omega_{1,1}(z_1) ={}& \frac{1}{48} (\zeta_1^{\mathrm{even}}(z_1) + 20 \, \zeta_0^{\mathrm{even}}(z_1)) \nonumber \\
={}& \frac{\dd z_1}{12 z_1} \frac{5z_1^8 - 8z_1^6 + 18z_1^4 - 8z_1^2 + 5}{(z_1^2-1)^4}
\end{align*}
\end{example}

A consequence of~\cref{lem:basis} is that elements of $V(z)$ are rational differential forms. The next result reveals that they possess interesting pole structure and symmetry.

\begin{proposition} \label{prop:vectorspace}
For all $\Omega(z) \in V(z)$,
\begin{itemize}
\item $\Omega(z)$ has poles only at $z = 1$, $z = -1$ and $z = 0$, with only simple poles occurring at $z=0$; and
\item $\Omega(z) + \Omega(\frac{1}{z}) = 0$.
\end{itemize}
\end{proposition}

\begin{proof}
The first statement is immediate from~\cref{lem:basis}, since the operator $\frac{\dd}{\dd z} (z \, \cdot)$ cannot introduce new poles on $\mathbb{CP}^1$. The second statement can be verified on the basis elements $\xi_k^{\mathrm{even}}(z)$ and $\xi_k^{\mathrm{odd}}(z)$, then deduced for all $\Omega(z) \in V(z)$ by linearity. The verification on basis elements is as follows, using the observation that $\tfrac{1}{z} \frac{\dd}{\dd (1/z)} = -z \frac{\dd}{\dd z}$.
\begin{align*}
\xi_k^{\mathrm{even}}(z) + \xi_k^{\mathrm{even}}(\tfrac{1}{z}) &= \dd \bigg[ \bigg( z \frac{\dd}{\dd z} \bigg)^{2k} \frac{z^2}{1-z^2} + \delta_{k,0} \log(z) \bigg] + \dd \bigg[ \bigg( -z \frac{\dd}{\dd z} \bigg)^{2k} \frac{(\tfrac{1}{z})^2}{1-(\tfrac{1}{z})^2} + \delta_{k,0} \log(\tfrac{1}{z}) \bigg] \\
&= \dd \bigg[ \bigg( z \frac{\dd}{\dd z} \bigg)^{2k} \bigg( \frac{z^2}{1-z^2} + \frac{1}{z^2-1} \bigg) + \delta_{k,0} \big( \log(z) + \log(\tfrac{1}{z}) \big) \bigg] = 0 \\ \\
\xi_k^{\mathrm{odd}}(z) + \xi_k^{\mathrm{odd}}(\tfrac{1}{z}) &= \dd \bigg[ \bigg( z \frac{\dd}{\dd z} \bigg)^{2k} \frac{z}{1-z^2} \bigg] + \dd \bigg[ \bigg( -z \frac{\dd}{\dd z} \bigg)^{2k} \frac{\tfrac{1}{z}}{1-(\tfrac{1}{z})^2} \bigg] \\
&= \dd \bigg[ \bigg( z \frac{\dd}{\dd z} \bigg)^{2k} \bigg( \frac{z}{1-z^2} + \frac{z}{z^2-1} \bigg) \bigg] = 0 \qedhere
\end{align*}
\end{proof}

We next state a lemma concerning $V(z)$ that will be necessary for the subsequent proof of~\cref{thm:TR}.

\begin{lemma} \label{lem:resatzero}
For all $\Omega(z) \in V(z)$,
\begin{equation}
\sum_{\alpha = \pm 1} \Res_{z=\alpha} \Omega(z) \log(z) = \Res_{z=0} \Omega(z).
\end{equation}
\end{lemma}

\begin{proof}
We simply verify the equation for the basis elements $\xi_k^{\mathrm{even}}(z)$ and $\xi_k^{\mathrm{odd}}(z)$, then deduce it for all $\Omega(z) \in V(z)$ by linearity. Note that the residue on the right side is 0 for each basis element, apart from $\xi_0^{\mathrm{even}}(z)$. So let us first suppose that $k \geq 1$ and verify the equation for $\xi_k^{\mathrm{even}}(z)$.
\begin{align*}
\sum_{\alpha = \pm 1} \Res_{z=\alpha} \xi_k^{\mathrm{even}}(z) \log(z) &= - \sum_{\alpha = \pm 1} \Res_{z=\alpha} \bigg[ \int \xi_k^{\mathrm{even}}(z) \bigg] \dd \log(z) = - \sum_{\alpha = \pm 1} \Res_{z=\alpha} \bigg[ \bigg( z \frac{\dd}{\dd z} \bigg)^{2k} \frac{z^2}{1-z^2} \bigg] \frac{\dd z}{z} \\
&= - \sum_{\alpha = \pm 1} \Res_{z=\alpha} \bigg[ \bigg( z \frac{\dd}{\dd z} \bigg)^{2k} \frac{1}{1-z^2} \bigg] \frac{\dd z}{z} \\
&= \Res_{z=0} \bigg[ \frac{\dd}{\dd z} \bigg( z \frac{\dd}{\dd z} \bigg)^{2k-1} \frac{1}{1-z^2} \bigg] \dd z
\end{align*}
The first line uses the fact that a function $F(z)$ that is meromorphic at $z = \alpha$ satisfies $\displaystyle \Res_{z=\alpha} \dd F = 0$. It follows that $\displaystyle \Res_{z=\alpha} f \, \dd g = - \displaystyle \Res_{z=\alpha} g \, \dd f$ for any two functions $f(z)$ and $g(z)$ that are meromorphic at $z = \alpha$. The second line uses the fact that $k$ is positive. The third line uses the fact that the sum of the residues of a rational differential form is equal to 0. It is clear that the final expression obtained is equal to 0, since the argument is holomorphic at $z = 0$. This completes the proof in this case.

The analogous calculation for $\xi_k^{\mathrm{odd}}(z)$ and $k \geq 0$ is almost identical to the previous and is omitted for brevity. It remains to treat the case $\xi_0^{\mathrm{even}}(z)$, in which case the residue on the right side of the equation is evidently equal to 1. We calculate the left side as follows.
\begin{align*}
\sum_{\alpha = \pm 1} \Res_{z=\alpha} \xi_0^{\mathrm{even}}(z) \log(z) &= \sum_{\alpha = \pm 1} \Res_{z=\alpha} \bigg[ \frac{\dd}{\dd z} \frac{z^2}{1-z^2} \, \dd z + \frac{\dd z}{z} \bigg] \log(z) = \sum_{\alpha = \pm 1} \Res_{z=\alpha} \bigg[ \frac{\dd}{\dd z} \frac{z^2}{1-z^2} \, \dd z \bigg] \log(z) \\
&= - \sum_{\alpha = \pm 1} \Res_{z=\alpha} \bigg[ \frac{z^2}{1-z^2} \bigg] \frac{\dd z}{z} = - \sum_{\alpha = \pm 1} \Res_{z=\alpha} \frac{z}{1-z^2} \, \dd z
\end{align*}
The first line uses the definition of $\xi_0^{\mathrm{even}}(z)$ and removes a summand that is clearly holomorphic at $z = \pm 1$. The second line uses the fact that $\displaystyle \Res_{z=\alpha} f \, \dd g = - \displaystyle \Res_{z=\alpha} g \, \dd f$ for any two functions $f(z)$ and $g(z)$ that are meromorphic at $z = \alpha$. It is then straightforward to calculate that the final expression obtained is equal to~1. This completes the proof in this case.
\end{proof}

\subsection{Combinatorial recursion}

We now express the combinatorial recursion of~\cref{thm:propNgn} in terms of natural generating functions. Rather than using the multidifferentials $\Omega_{g,n}(z_1, z_2, \ldots, z_n)$ defined earlier, it will be convenient to work with the closely related generating functions
\[
W_{g,n}(z_1, z_2, \ldots, z_n) = \frac{\Omega_{g,n}(z_1, z_2, \ldots, z_n)}{\dd z_1 \, \dd z_2 \, \cdots \, \dd z_n} = \sum_{b_1, b_2, \ldots, b_n=0}^\infty \N_{g,n}(b_1, b_2, \ldots, b_n) \prod_{i=1}^n [b_i] z_i^{b_i-1}.
\]

\begin{proposition} \label{prop:genfun}
For $2g-2+n \geq 2$, we have the following equation, where $S = \{1, 2, \ldots, n\}$ and $\zz_I = (z_{i_1}, z_{i_2}, \ldots, z_{i_k})$ for $I = \{i_1, i_2, \ldots, i_k\}$.
\begin{align*}
& \sum_{i=1}^n \frac{\partial}{\partial z_i} z_i W_{g,n}(\zz_S) = \sum_{i < j} \Bigg( \frac{\partial}{\partial z_i} \bigg[ \frac{2}{z_j} \frac{z_i^3}{(1-z_i^2)^2} W_{g,n-1}(\zz_{S\setminus\{j\}}) \bigg] + \frac{\partial}{\partial z_j} \bigg[ \frac{2}{z_i} \frac{z_j^3}{(1-z_j^2)^2} W_{g,n-1}(\zz_{S\setminus\{i\}}) \bigg] \\
& \qquad + 2 \frac{\partial}{\partial z_i} \frac{\partial}{\partial z_j} \bigg[ \frac{z_j}{z_i - z_j} \frac{z_i^3}{(1-z_i^2)^2} W_{g,n-1}(\zz_{S\setminus\{j\}}) - \frac{z_i}{z_i - z_j} \frac{z_j^3}{(1-z_j^2)^2} W_{g,n-1}(\zz_{S\setminus\{i\}}) \bigg] \Bigg) \\
& \qquad + \sum_{i=1}^{n} \frac{\partial}{\partial z_i} \frac{z_i^4}{(1-z_i^2)^2} \Bigg[ W_{g-1,n+1}(z_i, z_i, \zz_{S\setminus\{i\}})+ \sum_{\substack{g_1+g_2=g\\ I\sqcup J=S\setminus\{i\}}}^{\mathrm{stable}} W_{g_1,|I|+1}(z_i,\zz_I ) \, W_{g_2,|J|+1}(z_i, \zz_J ) \Bigg]
\end{align*}
\end{proposition}

\begin{proof}
The combinatorial recursion of~\cref{thm:propNgn} states that for $2g-2+n \geq 2$ and $b_1, b_2, \ldots, b_n \geq 0$, we have the following equation.
\begin{align*}
&\left(\sum_{i=1}^{n} b_i\right) \N_{g,n}(\bb _S) = \sum_{i < j} \sum_{\substack{p+q= b_i+b_j \\ q \text{ even}}} [p]q \, \N_{g,n-1}(p, \bb_{S\setminus \{i,j\}}) \\
&\qquad + \frac{1}{2} \sum_i \sum_{\substack{p+q+r=b_i \\ r \text{ even}}} [p] [q] r \Bigg[ \N_{g-1,n+1}(p,q, \bb_{S \setminus \{i\}}) + \sum_{\substack{g_1+g_2=g \\ I \sqcup J = S \setminus \{i\}}}^{\mathrm{stable}} \N_{g_1,|I|+1}(p,\bb_I) \, \N_{g_2,|J|+1}(q,\bb_J) \Bigg]
\end{align*}

Let us define the operators
\[
\mathcal{O} = \sum_{b_1,b_2, \ldots, b_n=0}^{\infty} [~ \cdot~ ] \, \prod_{i=1}^{n} \, [b_i] z_i^{b_i-1} \qquad \text{and} \qquad \mathcal{O}_J = \sum_{b_i=0: i \notin J}^{\infty} [~ \cdot~ ] \, \prod_{i \notin J} \, [b_i] z_i^{b_i-1}.
\]

The result arises from applying the operator $\mathcal O$ to both sides of the combinatorial recursion. The left side becomes
\begin{align*}
\sum_{b_1,b_2, \ldots, b_n=0}^{\infty} \left(\sum_{i=1}^{n} b_i\right) \N_{g,n}(\bb _S) \prod_{i=1}^{n} \, [b_i] z_i^{b_i-1} &= \sum_{i=1}^{n} \sum_{b_1,b_2, \ldots, b_n=0}^{\infty} \frac{\partial}{\partial z_i} z_i \bigg( \N_{g,n}(\bb _S) \prod_{i=1}^{n} \, [b_i] z_i^{b_i-1} \bigg) \\
&= \sum_{i=1}^{n} \frac{\partial}{\partial z_i} z_i W_{g,n}(\zz _S). \tag{$\ast$}
\end{align*}

Applying the operator $\mathcal{O}$ to the $(i,j)$th summand in the first term on the right side of the combinatorial recursion yields
\begin{align*}
& \sum_{b_1,b_2, \ldots, b_n=0}^{\infty} \sum_{\substack{p+q=b_i+b_j \\ q \text{ even}}} [p]q \, \N_{g,n-1}(p, \bb_{S\setminus\{i,j\}}) \, \prod_{i=1}^{n} \, [b_i] z_i^{b_i-1} \\
&\qquad = \mathcal{O}_{i,j} \sum_{b_i,b_j=0}^{\infty} \sum_{\substack{p+q=b_i+b_j \\ q \text{ even}}} [p]q \, \N_{g,n-1}(p, \bb_{S\setminus\{i,j\}}) \, [b_i] \, [b_j] z_i^{b_i-1} z_j^{b_j-1} \\
&\qquad = \mathcal{O}_{i,j} \sum_{\substack{p,q=0 \\ q \text{ even}}}^{\infty} [p] q \, \N_{g,n-1}(p, \bb_{S\setminus \{i,j\}}) \sum_{k=0}^{p+q} \, [k] \, [p+q-k] z_i^{k-1} z_j^{p+q-k-1} \\
&\qquad = \mathcal{O}_{i,j} \sum_{\substack{p,q=0 \\ q \text{ even}}}^{\infty} [p] q \, \N_{g,n-1}(p,\bb_{S\setminus \{i,j\}}) \bigg[ \frac{\partial}{\partial z_i} z_i^{p+q} z_j^{-1} + \frac{\partial}{\partial z_j} z_i^{-1} z_j^{p+q} \bigg] \\
& \qquad \quad + \mathcal{O}_{i,j} \sum_{\substack{p,q=0 \\ q \text{ even}}}^{\infty} [p] q \, \N_{g,n-1}(p,\bb_{S\setminus \{i,j\}}) \frac{\partial}{\partial z_i} \frac{\partial}{\partial z_j} \big( z_i^{p+q-1}z_j^{1} + z_i^{p+q-2}z_j^{2} + \cdots + z_i^{1} z_j^{p+q-1} \big). 
\end{align*}

Consider the first of the two terms in this last expression and use $\sum\limits_{q \text{ even}}q z^q = \frac{2z^2}{(1-z^2)^2}$ to obtain the following.
\begin{align*}
& \mathcal{O}_{i,j} \sum_{\substack{p,q=0 \\ q \text{ even}}}^{\infty} [p] q \, \N_{g,n-1}(p,\bb_{S\setminus \{i,j\}}) \bigg[ \frac{\partial}{\partial z_i} z_i^{p+q} z_j^{-1} + \frac{\partial}{\partial z_j} z_i^{-1} z_j^{p+q} \bigg] \\
&\qquad = \mathcal{O}_{i,j} \frac{\partial}{\partial z_i} z_j^{-1} z_i \sum_{\substack{q=0\\q\text{ even}}}^{\infty} q z_i^{q} \sum_{p=0}^{\infty} \, [p] \, \N_{g,n-1}(p,\bb_{S\setminus \{i,j\}}) \, z_i^{p-1} \\
& \qquad \quad + \mathcal{O}_{i,j} \frac{\partial}{\partial z_j} z_i^{-1} z_j \sum_{\substack{q=0\\q\text{ even}}}^{\infty} q z_j^{q} \sum_{p=0}^{\infty} \, [p] \, \N_{g,n-1}(p,\bb_{S\setminus \{i,j\}}) \, z_j^{p-1} \\
&\qquad = \frac{\partial}{\partial z_i} \bigg[ \frac{2}{z_j} \frac{z_i^3}{(1-z_i^2)^2} W_{g,n-1}(\zz_{S\setminus\{j\}}) \bigg] + \frac{\partial}{\partial z_j} \bigg[ \frac{2}{z_i} \frac{z_j^3}{(1-z_j^2)^2} W_{g,n-1}(\zz_{S\setminus\{i\}}) \bigg] \tag{$\ast$}
\end{align*}

Now consider the second term in a similar fashion to obtain the following.
\begin{align*}
& \mathcal{O}_{i,j}\sum_{\substack{p,q=0 \\ q \text{ even}}}^{\infty} [p] q \, \N_{g,n-1}(p,\bb_{S\setminus \{i,j\}}) \frac{\partial}{\partial z_i} \frac{\partial}{\partial z_j} \big( z_i^{p+q-1}z_j^{1} + z_i^{p+q-2}z_j^{2} + \cdots + z_i^{1} z_j^{p+q-1} \big) \\ 
&\qquad = \mathcal{O}_{i,j} \sum_{\substack{p,q=0 \\ q \text{ even}}}^{\infty} [p] q \, \N_{g,n-1}(p,\bb_{S\setminus \{i,j\}}) \frac{\partial}{\partial z_i} \frac{\partial}{\partial z_j} \frac{z_i^{p+q} z_j - z_i z_j^{p+q}}{z_i - z_j} \\
&\qquad = \mathcal{O}_{i,j} \frac{\partial}{\partial z_i} \frac{\partial}{\partial z_j} \bigg[ \frac{1}{z_i - z_j} \sum_{p,q=0}^{\infty} [p] q \, \N_{g,n-1}(p,\bb_{S\setminus \{i,j\}}) \big(z_i^{p+q} z_j - z_i z_j^{p+q} \big) \bigg] \\
&\qquad = 2 \frac{\partial}{\partial z_i} \frac{\partial}{\partial z_j} \bigg[ \frac{z_j}{z_i - z_j} \frac{z_i^3}{(1-z_i^2)^2} W_{g,n-1}(\zz_{S\setminus\{j\}}) - \frac{z_i}{z_i - z_j} \frac{z_j^3}{(1-z_j^2)^2} W_{g,n-1}(\zz_{S\setminus\{i\}}) \bigg] \tag{$\ast$}
\end{align*}

Applying the operator $\mathcal{O}$ to twice the $i$th summand in the second term on the right side of the combinatorial recursion yields
\begin{align*}
& \sum_{\substack{b_1, b_2, \ldots, b_n=0 \\ p+q+r=b_i \\ r \text{ even}}}^{\infty} [p][q]r \Bigg[\N_{g-1,n+1}(p,q,\bb_{S\setminus\{i\}}) + \sum_{\substack{g_1+g_2=g\\ I\sqcup J=S\setminus\{i\}}}^{\mathrm{stable}} \N_{g_1,|I|+1}(p,\bb_I) \, \N_{g_2,|J|+1}(q,\bb_J)\Bigg] \prod_{i=1}^{n} \, [b_i] z_i^{b_i-1} \\
&\quad = \mathcal{O}_i \sum_{\substack{b_i=0 \\ p+q+r=b_i \\ r \text{ even}}} [p][q]r \Bigg[\N_{g-1,n+1}(p,q,\bb_{S\setminus\{i\}})+ \sum_{\substack{g_1+g_2=g\\ I\sqcup J=S\setminus\{i\}}}^{\mathrm{stable}} \N_{g_1,|I|+1}(p,\bb_I) \, \N_{g_2,|J|+1}(q,\bb_J)\Bigg] [b_i] z_i^{b_i-1} \\
&\quad = \mathcal{O}_i \frac{\partial}{\partial z_i} z_i \sum_{\substack{p,q,r=0 \\ r \text{ even}}}^{\infty} [p][q]r \Bigg[\N_{g-1,n+1}(p,q,\bb_{S\setminus\{i\}})+ \sum_{\substack{g_1+g_2=g\\ I\sqcup J=S\setminus\{i\}}}^{\mathrm{stable}} \N_{g_1,|I|+1}(p,\bb_I) \, \N_{g_2,|J|+1}(q,\bb_J)\Bigg] z_i^{p+q+r-1} \\
&\quad = \frac{\partial}{\partial z_i} \frac{2z_i^4}{(1-z_i^2)^2} \Bigg[ W_{g-1,n+1}(z_i, z_i, \zz_{S\setminus\{i\}})+ \sum_{\substack{g_1+g_2=g\\ I\sqcup J=S\setminus\{i\}}}^{\mathrm{stable}} W_{g_1,|I|+1}(z_i, \zz_I) \, W_{g_2,|J|+1}(z_i, \zz_J) \Bigg] \tag{$\ast$}
\end{align*}

Finally, combine all of the contributions from the expressions marked by ($\ast$) to obtain the desired result.
\end{proof}

\subsection{Breaking the symmetry}

A feature of the topological recursion is that it produces symmetric meromorphic multidifferentials from a recursion that is manifestly asymmetric, with a special role played by the variable $z_1$. We break the symmetry in the recursion of~\cref{prop:genfun} by applying the operator 
\[
F(z_1) \mapsto F(z_1) - \frac{1}{z_1^2} F(\tfrac{1}{z_1})
\]
to every term appearing. In a precise sense, this amounts to taking the symmetric part with respect to the involution $s(z) = \frac{1}{z}$ appearing in the topological recursion, stated at the level of functions rather than differentials.

Recall that~\cref{eq:vectorspace} combined with~\cref{prop:vectorspace} assert that
\[
\Omega_{g,n}(z_1, z_2, \ldots, z_n) + \Omega_{g,n}(\tfrac{1}{z_1}, z_2, \ldots, z_n) = 0.
\]
At the level of generating functions, this translates to the property
\begin{equation} \label{eq:wsymmetry}
W_{g,n}(z_1, z_2, \ldots, z_n) - \frac{1}{z_1^2} W_{g,n}(\tfrac{1}{z_1}, z_2, \ldots, z_n) = 0,
\end{equation}
which will be useful in subsequent calculations.

\begin{proposition} \label{prop:asymmetricgf}
For $2g-2+n \geq 2$, we have the following equation, where $S = \{2, 3, \ldots, n\}$ and $\zz_I = (z_{i_1}, z_{i_2}, \ldots, z_{i_k})$ for $I = \{i_1, i_2, \ldots, i_k\}$.
\begin{align*}
&W_{g,n}(z_1, \zz _S) - \Res_{p=0} \frac{W_{g,n}(p, \zz _S) \, \dd p}{z_1} = \sum_{j = 2}^{n} \bigg( \frac{2}{z_1 z_j} + \frac{1}{(z_1-z_j)^2} + \frac{1}{(1 - z_1z_j)^2} \bigg) \frac{z_1^3}{(1-z_1^2)^2} W_{g,n-1}(z_1, \zz_{S \setminus \{j\}}) \\
&\qquad - \sum_{j=2}^n \frac{\partial}{\partial z_j} \bigg[ \bigg( \frac{1}{z_1-z_j} + \frac{z_j}{1 - z_1z_j} \bigg) \frac{z_j^3}{(1-z_j^2)^2} W_{g,n-1}(\zz _S) \bigg] \\
&\qquad + \frac{z_1^3}{(1-z_1^2)^2} \Bigg[ W_{g-1,n+1}(z_1, z_1, \zz_S) + \sum_{\substack{g_1+g_2=g \\ I \sqcup J = S}}^{\mathrm{stable}} W_{g_1, |I|+1} (z_1, \zz_I) \, W_{g_2,|J|+1}(z_1, \zz_J) \Bigg]
\end{align*}
\end{proposition}

\begin{proof}
As mentioned above, we apply the operator $F(z_1) \mapsto F(z_1) - \frac{1}{z_1^2} F(\tfrac{1}{z_1})$ to all terms appearing in the recursion of~\cref{prop:genfun}. The left side becomes
\begin{align*}
&\sum_{i=1}^n \frac{\partial}{\partial z_i} z_i W_{g,n}(z_1, \zz_S) - \frac{1}{z_1^2} \bigg[ \sum_{i=1}^n \frac{\partial}{\partial z_i} z_i W_{g,n}(z_1, \zz_S) \bigg]_{z_1 \mapsto \frac{1}{z_1}} \\
&\qquad = \frac{\partial}{\partial z_1} z_1 W_{g,n}(z_1, \zz_S) - \frac{1}{z_1^2} \frac{\partial}{\partial (\frac{1}{z_1})} \frac{1}{z_1} W_{g,n}(\tfrac{1}{z_1}, \zz_S) + \sum_{i=2}^n \frac{\partial}{\partial z_i} z_i \bigg[ W_{g,n}(z_1, \zz_S) - \frac{1}{z_1^2} W_{g,n}(\tfrac{1}{z_1}, \zz_S) \bigg] \\
&\qquad = 2 \frac{\partial}{\partial z_1} z_1 W_{g,n}(z_1, \zz_S). \tag{$\ast \ast$}
\end{align*}
Here, we have used the symmetry property of~\cref{eq:wsymmetry} to deduce that the summands with $2 \leq i \leq n$ are equal to 0 and to express $W_{g,n}(\tfrac{1}{z_1}, \zz_S)$ in terms of $W_{g,n}(z_1, \zz_S)$.

In the summation over $i < j$ on the right side, the symmetry property of~\cref{eq:wsymmetry} ensures that a non-zero contribution arises only for the summands with $i = 1$ and $j = 2, 3, \ldots, n$. For such a summand, the first line on the right side contributes
\begin{align*}
& \frac{\partial}{\partial z_1} \bigg[ \frac{2}{z_j} \frac{z_1^3}{(1-z_1^2)^2} W_{g,n-1}(\zz_{S\setminus\{j\}}) \bigg] - \frac{1}{z_1^2} \frac{\partial}{\partial (\frac{1}{z_1})} \bigg[ \frac{2}{z_j} \frac{\frac{1}{z_1^3}}{(1-\frac{1}{z_1^2})^2} W_{g,n-1}(\tfrac{1}{z_1}, \zz_{S\setminus\{1, j\}}) \bigg] \\
&+ \frac{\partial}{\partial z_j} \bigg[ \frac{1}{z_1} \frac{z_j^3}{(1-z_j^2)^2} W_{g,n-1}(\zz_S) \bigg] - \frac{1}{z_1^2} \frac{\partial}{\partial z_j} \bigg[ z_1 \frac{z_j^3}{(1-z_j^2)^2} W_{g,n-1}(\zz_S) \bigg] \\
&\qquad = 2 \frac{\partial}{\partial z_1} \bigg[ \frac{2}{z_j} \frac{z_1^3}{(1-z_1^2)^2} W_{g,n-1}(\zz_{S\setminus\{j\}}) \bigg]. \tag{$\ast \ast$}
\end{align*}

The second line on the right side contributes
\begin{align*}
& 2 \frac{\partial}{\partial z_1} \frac{\partial}{\partial z_j} \bigg[\frac{1}{z_1 - z_j} \frac{z_1^3 z_j}{(1-z_1^2)^2} W_{g,n-1}(z_1, \zz_{S\setminus\{j\}}) \bigg] - 2 \frac{1}{z_1^2} \frac{\partial}{\partial (\frac{1}{z_1})} \frac{\partial}{\partial z_j} \bigg[ \frac{1}{\frac{1}{z_1} - z_j} \frac{\frac{1}{z_1^3} z_j}{(1-\frac{1}{z_1^2})^2} W_{g,n-1}(\tfrac{1}{z_1}, \zz_{S\setminus\{j\}}) \bigg] \\
&- 2 \frac{\partial}{\partial z_1} \frac{\partial}{\partial z_j} \bigg[\frac{1}{z_1 - z_j} \frac{z_1 z_j^3}{(1-z_j^2)^2} W_{g,n-1}(\zz_S) \bigg] + 2 \frac{1}{z_1^2} \frac{\partial}{\partial (\frac{1}{z_1})} \frac{\partial}{\partial z_j} \bigg[\frac{1}{\frac{1}{z_1} - z_j} \frac{\frac{1}{z_1} z_j^3}{(1-z_j^2)^2} W_{g,n-1}(\zz_S) \bigg] \\
&\qquad = 2 \frac{\partial}{\partial z_1} \frac{\partial}{\partial z_j} \bigg[ \bigg( \frac{z_j}{z_1 - z_j} + \frac{z_1 z_j}{1 - z_1z_j} \bigg) \frac{z_1^3}{(1-z_1^2)^2} W_{g,n-1}(z_1, \zz_{S\setminus\{j\}}) \bigg] \\
&\qquad \quad - 2 \frac{\partial}{\partial z_1} \frac{\partial}{\partial z_j} \bigg[ \bigg( \frac{z_1}{z_1 - z_j} + \frac{1}{1 - z_1 z_j} \bigg) \frac{z_j^3}{(1-z_j^2)^2} W_{g,n-1}(\zz_S) \bigg]. \tag{$\ast \ast$}
\end{align*}

In the summation over $i$ on the right side, the symmetry property of~\cref{eq:wsymmetry} ensures that a non-zero contribution arises only for the summands with $i = 1$. So the third line on the right side contributes
\begin{align*}
& \frac{\partial}{\partial z_1} \frac{z_1^4}{(1-z_1^2)^2} \Bigg[ W_{g-1,n+1}(z_1, z_1, \zz_S) + \sum_{\substack{g_1+g_2=g\\ I\sqcup J=S}}^{\text{stable}} W_{g_1,|I|+1}(z_1, \zz_I) \, W_{g_2,|J|+1}(z_1, \zz_J) \Bigg] \\
&- \frac{1}{z_1^2} \frac{\partial}{\partial (\tfrac{1}{z_1})} \frac{\frac{1}{z_1^4}}{(1-\frac{1}{z_1^2})^2} \Bigg[ W_{g-1,n+1}(\tfrac{1}{z_1}, \tfrac{1}{z_1}, \zz_S) + \sum_{\substack{g_1+g_2=g\\ I\sqcup J=S}}^{\text{stable}} W_{g_1,|I|+1}(\tfrac{1}{z_1}, \zz_I) \, W_{g_2,|J|+1}(\tfrac{1}{z_1}, \zz_J) \Bigg] \\
&\qquad = 2 \frac{\partial}{\partial z_1} \frac{z_1^4}{(1-z_1^2)^2} \Bigg[ W_{g-1,n+1}(z_1, z_1, \zz_S) + \sum_{\substack{g_1+g_2=g\\ I\sqcup J=S}}^{\text{stable}} W_{g_1,|I|+1}(z_1, \zz_I) \, W_{g_2,|J|+1}(z_1, \zz_J) \Bigg]. \tag{$\ast \ast$}
\end{align*}

Gather together all of the terms marked by ($\ast \ast$) and perform some mild algebraic simplification to obtain the following.
\begin{align*}
& \frac{\partial}{\partial z_1} z_1 W_{g,n}(z_1, \zz_S) = \sum_{j=2}^n \frac{\partial}{\partial z_1} \bigg[ z_1 \bigg( \frac{2}{z_1 z_j} + \frac{1}{(z_1 - z_j)^2} + \frac{1}{(1 - z_1z_j)^2} \bigg) \frac{z_1^3}{(1-z_1^2)^2} W_{g,n-1}(z_1, \zz_{S\setminus\{j\}}) \bigg] \\
&\qquad - \sum_{j=2}^n \frac{\partial}{\partial z_1} \frac{\partial}{\partial z_j} \bigg[ z_1 \bigg( \frac{1}{z_1 - z_j} + \frac{z_j}{1 - z_1 z_j} \bigg) \frac{z_j^3}{(1-z_j^2)^2} W_{g,n-1}(\zz_S) \bigg] \\
&\qquad + \frac{\partial}{\partial z_1} \frac{z_1^4}{(1-z_1^2)^2} \Bigg[ W_{g-1,n+1}(z_1, z_1, \zz_S) + \sum_{\substack{g_1+g_2=g\\ I\sqcup J=S}}^{\text{stable}} W_{g_1,|I|+1}(z_1, \zz_I) \, W_{g_2,|J|+1}(z_1, \zz_J) \Bigg]
\end{align*}

One can remove the operator $\frac{\partial}{\partial z_1} z_1$ from every term to recover an equality of the following form, where $[\mathrm{correction}]$ is independent of $z_1$.
\begin{align*}
& W_{g,n}(z_1, \zz_S) + \frac{[\mathrm{correction}]}{z_1} = \sum_{j=2}^n \bigg[ \bigg( \frac{2}{z_1 z_j} + \frac{1}{(z_1 - z_j)^2} + \frac{1}{(1 - z_1z_j)^2} \bigg) \frac{z_1^3}{(1-z_1^2)^2} W_{g,n-1}(z_1, \zz_{S\setminus\{j\}}) \bigg] \\
&\qquad - \sum_{j=2}^n \frac{\partial}{\partial z_j} \bigg[ \bigg( \frac{1}{z_1 - z_j} + \frac{z_j}{1 - z_1 z_j} \bigg) \frac{z_j^3}{(1-z_j^2)^2} W_{g,n-1}(\zz_S) \bigg] \\
&\qquad + \frac{z_1^3}{(1-z_1^2)^2} \Bigg[ W_{g-1,n+1}(z_1, z_1, \zz_S) + \sum_{\substack{g_1+g_2=g\\ I\sqcup J=S}}^{\text{stable}} W_{g_1,|I|+1}(z_1, \zz_I) \, W_{g_2,|J|+1}(z_1, \zz_J) \Bigg]
\end{align*}

Finally, recall that $W_{g,n}(z_1, z_2, \ldots, z_n)$ has at worst a simple pole at $z_1 = 0$. It follows that the right side of this equation has no pole at $z_1=0$, so the correction term is given by
\[
[\mathrm{correction}] = - \Res_{p=0} W_{g,n}(p, \zz_S) \, \dd p,
\]
and this completes the proof.
\end{proof}

We remark that equating coefficients of $z_1, z_2, \ldots, z_n$ in the recursion of~\cref{prop:asymmetricgf} leads to the following asymmetric form of the combinatorial recursion for $\N_{g,n}(b_1, b_2, \ldots, b_n)$.

\begin{proposition}
For $2g-2+n \geq 2$ and $b_1, b_2, \ldots, b_n \geq 0$, we have the following equation, where $S = \{2, 3, \ldots, n\}$ and $\bb_I = (b_{i_1}, b_{i_2}, \ldots, b_{i_k})$ for $I = \{i_1, i_2, \ldots, i_k\}$.
\begin{align*}
2 b_1 \N_{g,n}(b_1, b _S)&= \sum_{j=2}^n \Bigg[ \sum_{p+q= b_1+b_j} [p]q \, \N_{g,n-1}(p, \bb_{S\setminus\{j\}}) + \sgn(b_1-b_j) \sum_{p+q= |b_1-b_j|} [p]q \, \N_{g,n-1}(\bb_{S\setminus\{j\}}) \Bigg] \\ 
&\quad + \sum_{p+q+r=b_1}[p][q]r\Bigg[ \N_{g-1,n+1}(p,q, \bb _S)+\sum_{\substack{g_1+g_2=g\\ I\sqcup J=S}}^{\mathrm{stable}} \N_{g_1,|I|+1}(p, \bb_I) \, \N_{g_2,|J|+1}(q, \bb_J)\Bigg]
\end{align*}
\end{proposition}

\subsection{Proof of the main theorem}

We now have all of the pieces in place to prove our main result.

\begin{proof}[Proof of~\cref{thm:TR}]
Recall that we wish to prove that $\Omega_{g,n} = \omega_{g,n}$ for all $(g,n) \neq (0,1)$ or $(0,2)$, where the former is defined via the enumeration $\N_{g,n}(b_1, b_2, \ldots, b_n)$ and~\cref{eq:defOmega}, while the latter is defined via the topological recursion applied to the local spectral curve of~\cref{eq:spectralcurve}. We use an inductive approach with base cases $(g,n) = (0,3)$ and $(1,1)$. One may verify directly that $\Omega_{0,3} = \omega_{0,3}$ and $\Omega_{1,1} = \omega_{1,1}$ by comparing the calculations of~\cref{ex:omegabases} and~\cref{ex:Omegabases}. So the theorem is true whenever $2g-2+n = 1$. Now consider $(g,n)$ satisfying $2g-2+n \geq 2$ and assume the inductive hypothesis that $\Omega_{g',n'} = \omega_{g',n'}$ whenever $1 \leq 2g'-2+n' < 2g-2+n$ and $(g',n') \neq (0,1)$ or $(0,2)$.

We begin by rewriting~\cref{prop:asymmetricgf} in terms of multidifferentials by multiplying by $\dd z_1 \, \dd z_2 \, \cdots \, \dd z_n$.
\begin{align} \label{eq:eq2}
& \Omega_{g,n}(z_1,\zz_S) - \frac{\dd z_1}{z_1} \Res_{p=0} \Omega_{g,n}(p, \zz _S) = \sum_{j = 2}^{n} \bigg( \frac{2 \, \dd z_j}{z_1 z_j} + \frac{\dd z_j}{(z_1-z_j)^2} + \frac{\dd z_j}{(1 - z_1z_j)^2} \bigg) \frac{z_1^3}{(1-z_1^2)^2} \, \Omega_{g,n-1}(z_1, \zz_{S \setminus \{j\}}) \nonumber \\
&\qquad - \sum_{j=2}^n \frac{\partial}{\partial z_j} \bigg[ \bigg( \frac{1}{z_1-z_j} + \frac{z_j}{1 - z_1z_j} \bigg) \frac{z_j^3}{(1-z_j^2)^2} W_{g,n-1}(\zz _S) \bigg] \dd z_1 \, \dd z_2 \, \cdots \, \dd z_n \nonumber \\
&\qquad + \frac{z_1^3}{(1-z_1^2)^2} \frac{1}{\dd z_1} \Bigg[ \Omega_{g-1,n+1}(z_1, z_1, \zz_S) + \sum_{\substack{g_1+g_2=g \\ I \sqcup J = S}}^{\mathrm{stable}} \Omega_{g_1, |I|+1} (z_1, \zz_I) \, \Omega_{g_2,|J|+1}(z_1, \zz_J) \Bigg]
\end{align}

By~\cref{prop:vectorspace}, $\Omega_{g,n}(z_1, \zz_S)$ has at worst a simple pole at $z_1=0$ and poles at $z_1 =1$ and $z_1 = -1$. Hence, the left side of the previous equation only has poles at $z_1=1$ and $z_1 = -1$. Now use the fact that a rational differential is equal to the sum of its principal parts, each of which may be expressed by~\cref{eq:principalpart}, to obtain the following.
\[
\Omega_{g,n}(z_1, \zz_S) - \frac{\dd z_1}{z_1} \Res_{p=0} \Omega_{g,n}(p, \zz _S) = \sum_{\alpha = \pm 1} \Res_{z=\alpha} \frac{\dd z_1}{z_1-z} \bigg[ \Omega_{g,n}(z, \zz _S) - \frac{\dd z}{z} \Res_{p=0} \Omega_{g,n}(p, \zz_S) \bigg]
\]

Substituting~\cref{eq:eq2} into the right side of the previous equation yields the following.
\begin{align} \label{eq:eq3}
&\Omega_{g,n}(z_1, \zz _S) - \frac{\dd z_1}{z_1} \Res_{p=0} \Omega_{g,n}(p, \zz _S) \nonumber \\
&\qquad = \sum_{\alpha = \pm 1} \Res_{z=\alpha} \frac{1}{z_1-z} \frac{z^3}{(1-z^2)^2} \frac{\dd z_1}{\dd z} \Bigg[ \sum_{j = 2}^{n} \bigg( \frac{2 \, \dd z \, \dd z_j }{zz_j} + \frac{\dd z \, \dd z_j}{(z-z_j)^2} + \frac{\dd z \, \dd z_j}{(1 - zz_j)^2} \bigg) \omega_{g,n-1}(z, \zz_{S \setminus \{j\}}) \nonumber \\
&\qquad \quad + \omega_{g-1,n+1}(z,z, \zz_S) + \sum_{\substack{g_1+g_2=g \\ I \sqcup J = S}}^{\mathrm{stable}} \omega_{g_1, |I|+1} (z, \zz_I) \, \omega_{g_2,|J|+1}(z, \zz_J) \Bigg]
\end{align}
Since the entire second line on the right side of~\cref{eq:eq2} is evidently analytic at $z_1 = \alpha$ for all $\alpha \in \mathbb{C}$, we may omit it from this equation. Furthermore, we have invoked the induction hypothesis to replace each $\Omega_{g',n'}$ on the right side with $\omega_{g',n'}$.

Recalling the definition of $\omega_{0,2}$, we have
\[
\omega_{0,2}(z, z_2) = \frac{\dd z \, \dd z_2}{(z - z_2)^2} + \frac{\dd z \, \dd z_2}{z\, z_2} \quad \Rightarrow \quad \omega_{0,2}(z, z_2) - \omega_{0,2}(\tfrac{1}{z}, z_2) = \frac{2 \, \dd z \, \dd z_2}{z\, z_2} + \frac{\dd z \, \dd z_2}{(z - z_2)^2} + \frac{\dd z \, \dd z_2}{(1 - z z_2)^2}.
\]
Therefore,~\cref{eq:eq3} above can be written equivalently as follows, where we have also used the induction hypothesis and~\cref{prop:vectorspace} to deduce that $\omega_{g',n'}(z, \zz) = - \omega_{g',n'}(\tfrac{1}{z}, \zz)$ for various terms on the right side.
\begin{align*}
& \Omega_{g,n}(z_1, \zz _S) = \frac{\dd z_1}{z_1} \Res_{p=0} \Omega_{g,n}(p, \zz _S) + \sum_{\alpha = \pm 1} \Res_{z=\alpha} \frac{-1}{z_1-z} \frac{z^3}{(1-z^2)^2} \frac{\dd z_1}{\dd z} \Bigg[ \sum_{j=2}^n \omega_{0,2}(z, z_2) \, \omega_{g,n-1}(\tfrac{1}{z}, \zz_{S \setminus \{j\}}) \\
& \qquad + \sum_{j=2}^n \omega_{0,2}(\tfrac{1}{z}, z_2) \, \omega_{g,n-1}(z, \zz_{S \setminus \{j\}}) + \omega_{g-1,n+1}(z,\tfrac{1}{z}, \zz_S) + \sum_{\substack{g_1+g_2=g \\ I \sqcup J = S}}^{\mathrm{stable}} \omega_{g_1, |I|+1}(z, \zz_I) \, \omega_{g_2,|J|+1}(\tfrac{1}{z}, \zz_J) \Bigg],
\end{align*}

Now absorb the terms in the two summations over $j$ into the summation over $g_1+g_2$ and $I\sqcup J = S$. Recall that the symbol $\circ$ over the inner summation denotes that we exclude all terms with $\omega_{0,1}$. 
\begin{align} \label{eq:eq4}
\Omega_{g,n}(z_1, \zz _S) &= \frac{\dd z_1}{z_1} \Res_{p=0} \Omega_{g,n}(p, \zz _S) + \sum_{\alpha = \pm 1} \Res_{z=\alpha} \frac{-1}{z_1-z} \frac{z^3}{(1-z^2)^2} \frac{\dd z_1}{\dd z} \bigg[ \omega_{g-1,n+1}(z,\tfrac{1}{z}, \zz_S) \nonumber \\
& \quad + \sum_{\substack{g_1+g_2=g \\ I \sqcup J = S}}^\circ \omega_{g_1, |I|+1} (z, \zz_I) \, \omega_{g_2,|J|+1}(\tfrac{1}{z}, \zz_J) \bigg].
\end{align}

By construction we have $\Omega_{g,n}(p, \zz_S) \in V(p) \otimes V(z_2) \otimes \cdots \otimes V(z_n)$, so~\cref{lem:resatzero} asserts that
\[
\Res_{p=0} \Omega_{g,n}(p, \zz_S) = \sum_{\alpha = \pm 1} \Res_{z=\alpha} \Omega_{g,n}(z, \zz_S) \log(z).
\]
Multiply both sides of this equation by $\frac{\dd z_1}{z_1}$ and use~\cref{eq:eq2} to substitute for $\Omega_{g,n}(z, \zz_S)$ on the right side. Observing that the terms
\[
\frac{\dd z_1}{z_1} \Res_{p=0} \Omega_{g,n}(p, \zz _S) \qquad \text{and} \qquad \sum_{j=2}^n \frac{\partial}{\partial z_j} \bigg[ \bigg( \frac{1}{z_1-z_j} + \frac{z_j}{1 - z_1z_j} \bigg) \frac{z_j^3}{(1-z_j^2)^2} W_{g,n-1}(\zz _S) \bigg]
\]
are analytic at $z_1 = 1$ and $z_1 = -1$, we obtain the following.
\begin{align*}
& \frac{\dd z_1}{z_1} \Res_{p=0} \Omega_{g,n}(p, \zz_S) \\
&\qquad = \sum_{\alpha = \pm 1} \Res_{z=\alpha} \Omega_{g,n}(z, \zz_S) \frac{\log(z)}{z_1} \dd z_1 \\
&\qquad = \sum_{\alpha = \pm 1} \Res_{z=\alpha} \frac{\log(z)}{z_1} \frac{z^3}{(1-z^2)^2} \frac{\dd z_1}{\dd z} \Bigg[ \sum_{j = 2}^{n} \left( \frac{2 \, \dd z \, \dd z_j}{z z_j} + \frac{\dd z \, \dd z_j}{(z - z_j)^2} + \frac{\dd z \, \dd z_j}{(1 - z z_j)^2} \right) \Omega_{g,n-1}(z, \zz_{S\setminus \{j\}}) \\
 & \qquad \quad + \Omega_{g-1,n+1}(z_1, z_1, \zz _S) + \sum_{\substack{g_1+g_2=g \\ I \sqcup J = S}}^{\mathrm{stable}} \Omega_{g_1, |I|+1} (z_1, \zz_I) \, \Omega_{g_2,|J|+1}(z_1, \zz_J) \Bigg] \\
&\qquad = \sum_{\alpha = \pm 1} \Res_{z=\alpha} \frac{- \log(z)}{\dd z} \frac{z^3}{(1-z^2)^2} \frac{\dd z_1}{\dd z} \Bigg[ \omega_{g-1,n+1}(z,\tfrac{1}{z}, \zz _S) + \sum_{\substack{g_1+g_2=g \\ I \sqcup J = S}}^\circ \omega_{g_1, |I|+1} (z, \zz_I) \, \omega_{g_2,|J|+1}(\tfrac{1}{z}, \zz_J) \Bigg]
\end{align*}
Here, we have used the induction hypothesis and the same algebraic trickery that was used previously to deduce~\cref{eq:eq4} from~\cref{eq:eq3}.

Substituting the previous equation into~\cref{eq:eq4} results in
\begin{align*}
\Omega_{g,n}(z_1, \zz _S) &= \sum_{\alpha = \pm 1} \Res_{z=\alpha} K(z_1, z) \Bigg[ \omega_{g-1,n+1}(z,\tfrac{1}{z}, \zz _S) + \sum_{\substack{g_1+g_2=g \\ I \sqcup J = S}}^\circ \omega_{g_1, |I|+1} (z, \zz_I) \, \omega_{g_2,|J|+1}(\tfrac{1}{z}, \zz_J) \Bigg],
\end{align*}
where we have recognised the recursion kernel $K(z_1, z)$ calculated in~\cref{ex:omegabases}. The right side of this equation coincides precisely with the topological recursion as defined in~\cref{sec:TR}, so we have finally deduced that $\Omega_{g,n} = \omega_{g,n}$. By induction, we conclude that $\Omega_{g,n} = \omega_{g,n}$ for all $(g,n) \neq (0,1)$ or $(0,2)$.
\end{proof}

\section{Applications and remarks} \label{sec:applications}

\subsection{String and dilaton equations}

The correlation differentials produced by the topological recursion satisfy {\em string and dilaton equations}~\cite{eyn-ora07}.
\begin{align}
\sum_\alpha \Res_{z=\alpha} y(z) \, \omega_{g,n+1}(z, \zz_S) &= - \sum_{k=1}^n \dd z_k \frac{\partial}{\partial z_k} \bigg( \frac{\omega_{g,n}(\zz_S)}{\dd x(z_k)} \bigg) \label{eq:stringeq} \\
\sum_\alpha \Res_{z=\alpha} \Phi(z) \, \omega_{g,n+1}(z, \zz_S) &= (2g-2+n) \, \omega_{g,n}(\zz_S) \label{eq:dilaton}
\end{align}
Each left side is a summation over the zeroes $\alpha$ of $\dd x$, $S$ denotes the set $\{1, 2, \ldots, n\}$, and $\Phi(z)$ is any function satisfying $\dd \Phi(z) = y(z) \, \dd x(z)$. Although these were originally proven in the context of global topological recursion, we show below that they also hold for the spectral curve of~\cref{thm:TR}. In that case, we immediately obtain the relations of~\cref{cor:stringdilaton}, which are known due to the previous work of the second author and Norbury~\cite{do-nor11}.

\begin{proof}[Proof of~\cref{cor:stringdilaton}]
First, we deal with the string equation. Consider the left side of~\cref{eq:stringeq} and use the fact that the sum of the residues at the poles of $y(z) \, \omega_{g,n+1}(z, \zz_S)$ is 0. Multiplying $\omega_{g,n+1}(z, \zz_{S})$ by $y(z) = z$ removes the simple pole and introduces a pole at $z = \infty$. So using~\cref{prop:vectorspace}, we have
\begin{align*}
\sum_{\alpha=\pm 1} \Res_{z=\alpha} y(z) \, \omega_{g,n+1}(z, \zz_S) &= - \Res_{z=\infty} z \, \omega_{g,n+1}(z, \zz_S) = -\Res_{z=0} \frac{1}{z} \omega_{g,n+1} (\tfrac{1}{z}, \zz_S) = \Res_{z=0} \frac{1}{z} \omega_{g,n+1}(z, \zz_S) \\
&= \sum_{b_1, b_2, \ldots, b_n=0}^{\infty} \N_{g,n+1}(1, \bb_S) \, \prod_{i=1}^n \, [b_i] z_i^{b_i-1} \, \dd z_i.
\end{align*}

Next, consider the $k$th summand of the right side of~\cref{eq:stringeq}.
\begin{align*}
& - \dd z_k \frac{\partial}{\partial z_k} \bigg( \frac{\omega_{g,n}(\zz_S)}{\dd x(z_k)} \bigg) = \dd z_k \frac{\partial}{\partial z_k} \bigg( \frac{1}{\dd z_k} \frac{z_k^2}{1-z_k^2} \sum_{b_1, b_2, \ldots, b_n=0}^{\infty} \N_{g,n}(\bb_S) \, \prod_{i=1}^n \, [b_i] z_i^{b_i-1} \, \dd z_i \bigg) \\
&\qquad = \dd z_k \sum_{a=0}^{\infty} \sum_{m=1}^{\infty} \sum_{b_1, \ldots, \widehat{b}_k, \ldots, b_n=0}^{\infty} \N_{g,n}(a, \bb_{S \setminus \{k\}}) \, [a] (a+2m-1) z_i^{a+2m-2} \, \prod_{i\in S\setminus \{k\}} [b_i] z_i^{b_i-1} \, \dd z_i
\end{align*}

Hence, extracting the coefficient of $\prod_{i=1}^{n} [b_i] \, z_i^{b_i-1} \, \dd z_i$ from the two sides of~\cref{eq:stringeq} leads to the first relation of~\cref{cor:stringdilaton}.
\[
\N_{g,n+1}(1, b_1, b_2, \ldots, b_n) = \sum_{k=1}^{n} \sum_{a=0}^{b_k} [a] \, \N_{g,n}(a, b_1, \ldots, \widehat{b}_k, \ldots, b_n)
\]

Next, we deal with the dilaton equation, in which case we take $\Phi(z) = \tfrac{1}{2} z^{2} - \log(z)$. Consider the left side of~\cref{eq:dilaton} and use~\cref{lem:resatzero} to deal with the logarithmic term that arises.
\begin{align*}
\sum_{\alpha=\pm 1} \Res_{z=\alpha} \Phi (z) \, \omega_{g,n+1}(z, \zz_S) &= \sum_{\alpha=\pm 1} \Res_{z=\alpha} \frac{1}{2} z^{2} \omega_{g,n+1}(z, \zz_S) - \sum_{\alpha=\pm 1} \Res_{z=\alpha} \log(z) \, \omega_{g,n+1}(z, \zz_S) \\
&= -\Res_{z=\infty} \frac{1}{2} z^{2} \omega_{g,n+1}(z, \zz_S) - \Res_{z=0} \omega_{g,n+1}(z, \zz_S) \\
&= \Res_{z=0} \frac{1}{2 z^{2}}\omega_{g,n+1}(z, \zz_S)-\Res_{z=0} \omega_{g,n+1}(z, \zz_S) \\
&= \sum_{b_1, b_2, \ldots, b_n=0}^\infty \Big[ \N_{g,n+1}(2, \bb_S) - \N_{g,n+1}(0, \bb_S) \Big] \prod_{i=1}^n \, [b_i] z_i^{b_i-1} \, \dd z_i
 \end{align*}

So extracting the coefficient of of $\prod_{i=1}^n \, [b_i] z_i^{b_i-1} \, \dd z_i$ from the two sides of~\cref{eq:dilaton} leads to the second relation of~\cref{cor:stringdilaton}.
\[
\N_{g,n+1}(2, b_1, b_2, \ldots, b_n) - \N_{g,n+1}(0, b_1, b_2, \ldots, b_n) = (2g - 2 + n) \, \N_{g,n}(b_1, b_2, \ldots, b_n) \qedhere
\]
\end{proof}

\subsection{Quantum curves}

The notion of topological recursion is closely related to the notion of quantum curve~\cite{nor16}. Briefly speaking, one integrates the correlation differentials and stores them in the following so-called {\em wave function}.
\[
\psi(x, \h) = \exp \left[ \sum_{g=0}^\infty \sum_{n=1}^\infty \frac{\h^{2g-2+n}}{n!} \int^x \!\! \int^x \!\!\cdots\!\! \int^x \omega_{g,n}(z_1, z_2, \ldots, z_n) \right]
\]
The wave function satisfies differential equations of the form
\[
\hat{P}(\hat{x}, \hat{y}) \, \psi(x, \h) = 0,
\]
where $\hat{x} = x$, $\hat{y} = -\h \frac{\partial}{\partial x}$ and $\hat{P}$ is a non-commutative polynomial. It has been empirically observed and proved in a variety of contexts that there is natural choice of $\hat{P}(\hat{x}, \hat{y})$ whose semi-classical limit $P(x, y) = 0$ recovers the underlying spectral curve for the topological recursion. Of course, this phenomenon most naturally applies to the case of global spectral curves. As an example, it is known that the enumeration of lattice points in ${\mathcal M}_{g,n}$ is governed by the global rational spectral curve $x(z) = z + \frac{1}{z}$ and $y(z) = z$ and that the corresponding quantum curve is given by the following operator $\hat{P}(\hat{x}, \hat{y}) = \hat{y}^2 - \hat{x} \hat{y} + 1$~\cite{do-man14}.

It would be interesting to construct a natural wave function for the topological recursion of~\cref{thm:TR} and to find a quantum curve operator that annihilates it. Although the spectral curve is not global in the usual sense, it has the same underlying algebraic curve as for the enumeration of lattice points in ${\mathcal M}_{g,n}$. Thus, one might expect a different quantum curve operator to the one above, which still recovers $y^2 - xy + 1 = 0$ in the semi-classical limit. Examples of  this nature may help to shed further light on the still mysterious phenomenon of quantum curves.

\subsection{Where did the spectral curve come from?}

It is natural to ask where the spectral curve of~\cref{thm:TR} came from. In particular, it would be useful to be able to identify other problems that are governed by local topological recursion, perhaps with a modified $\omega_{0,2}$ as in the case here. Typically, one can speculate the form of a global spectral curve attached to an enumerative problem from the case $(g,n) = (0,1)$, given that $\omega_{0,1}(z_1) = -y(z_1) \, \dd x(z_1)$. The enumeration of lattice points in $\overline{\mathcal M}_{g,n}$ for $(g,n) = (0,1)$ and $(0,2)$ matches the enumeration of lattice points in ${\mathcal M}_{g,n}$, which indicates using the same $x(z)$ and $y(z)$ in the spectral curve data.\footnote{This statement is somewhat subtle, since the natural definitions would lead to $\N_{0,1}(b) = 0$ for $b > 0$. Instead, consider the enumeration of lattice points in ${\mathcal M}_{g,n}$ and $\overline{\mathcal M}_{g,n}$ as the enumeration of ordinary and stable ribbon graphs, in which all vertices have degree at least two. One can pass to the analogous problems in which this degree condition is relaxed using the pruning correspondence~\cite{do-nor18}. The resulting problems are stored in the same correlation differentials, but as coefficients in the expansion at $x = \infty$, rather than at $z = 0$. It is the alignment of these problems for $(g,n) = (0,1)$ and $(0,2)$ that suggests using the same $x(z)$ and $y(z)$ in the spectral curve data.}

The spectral curve of~\cref{thm:TR} arises from a modification to $\omega_{0,2}$ for the enumeration of lattice points in ${\mathcal M}_{g,n}$. One moves to the compactified version of the count by allowing nodes and there is a sense in which nodes correspond to $(0,2)$ information. For example, the stabilisation of a nodal curve contracts $(0,2)$ components --- that is, components with genus zero and two nodal points --- to nodes. Alternatively, consider the graphical interpretation of topological recursion, which expresses each correlation differential $\omega_{g,n}$ as a weighted sum over decorated graphs~\cite{dun-ora-sha-spi14,eyn-ora07}. For each such graph, the vertices are weighted by intersection numbers on $\overline{\mathcal M}_{g,n}$ and the edges by so-called {\em jumps}, which are essentially the coefficients of $\omega_{0,2}$. These decorated graphs bear a close relation to the graphs arising from the stratification of $\overline{\mathcal M}_{g,n}$, so that edges correspond to nodes. So again, we see that $\omega_{0,2}$ controls nodal behaviour and it should come as less of a surprise that the enumeration of lattice points in $\overline{\mathcal M}_{g,n}$ requires a modification to $\omega_{0,2}$. That the extra contribution to $\omega_{0,2}$ is of the form $\frac{\dd z_1 \, \dd z_2}{z_1 z_2}$ corresponds to the fact that we should take $\N_{0,2}(0,0) = 1$.

It would be interesting to take standard enumerative problems governed by global topological recursion --- such as the psi-class intersection numbers on ${\mathcal M}_{g,n}$, simple Hurwitz numbers and the Gromov--Witten theory of $\mathbb{CP}^1$ --- and consider the effect of a modification to $\omega_{0,2}$ on the associated correlation differentials.

\appendix

\section{Data} \label{app:data}

\cref{thm:propNgn} asserts that $\N_{g,n}(b_1, b_2, \ldots, b_n)$ is a symmetric quasi-polynomial that is non-zero only when $b_1 + b_2 + \cdots + b_n$ is even. Hence, $\N_{g,n}(b_1, b_2, \ldots, b_n)$ can be described by the underlying polynomials $\N_{g,n}^{(k)}(b_1, b_2, \ldots, b_n)$ that determine it in the case $b_1, b_2, \ldots, b_k$ are odd and $b_{k+1}, b_{k+2}, \ldots, b_n$ are even, where we may restrict to $k$ even. The following table is replicated from the literature~\cite{do-nor11} and stores this information for some small values of $g$ and~$n$. We remind the reader that the data lends strong support to~\cref{con:positivity}, which speculates that the coefficients of such polynomials are always non-negative.

\begin{center}
\begin{tabularx}{\textwidth}{cccX} \toprule
$g$ & $n$ & $k$ & $\N_{g,n}^{(k)}(b_1, b_2, \ldots, b_n)$ \\ \midrule
0 & 3 & 0 & 1 \\
0 & 3 & 2 & 1 \\
1 & 1 & 0 & $\frac{1}{48} (b_1^2 + 20)$ \\
0 & 4 & 0 & $\frac{1}{4}(b_1^2 + b_2^2 + b_3^2 + b_4^2 + 8)$ \\
0 & 4 & 2 & $\frac{1}{4}(b_1^2 + b_2^2 + b_3^2 + b_4^2 + 2)$ \\
0 & 4 & 4 & $\frac{1}{4}(b_1^2 + b_2^2 + b_3^2 + b_4^2 + 8)$ \\
1 & 2 & 0 & $\frac{1}{384}(b_1^4 + b_2^4 + 2b_1^2b_2^2 + 36b_1^2 + 36b_2^2 + 192)$ \\
1 & 2 & 2 & $\frac{1}{384}(b_1^4 + b_2^4 + 2b_1^2b_2^2 + 36b_1^2 + 36b_2^2 + 84)$ \\
0 & 5 & 0 & $\frac{1}{32} \sum b_i^4 + \frac{1}{8} \sum b_i^2 b_j^2 + \frac{7}{8} \sum b_i^2 + 7$ \\
0 & 5 & 2 & $\frac{1}{32} \sum b_i^4 + \frac{1}{8} \sum b_i^2 b_j^2 + \frac{5}{16}(b_1^2+b_2^2) + \frac{1}{8}(b_3^2+b_4^2+b_5^2) + \frac{19}{16}$ \\
0 & 5 & 4 & $\frac{1}{32} \sum b_i^4 + \frac{1}{8} \sum b_i^2 b_j^2 + \frac{5}{16} (b_1^2+b_2^2+b_3^2+b_4^2) + \frac{7}{8}b_5^2 + \frac{7}{8}$ \\
1 & 3 & 0 & $\frac{1}{4608} \sum b_i^6 + \frac{1}{768} \sum b_i^4 b_j^2 + \frac{1}{384} b_1^2b_2^2b_3^2 + \frac{13}{1152} \sum b_i^4 + \frac{1}{24} \sum b_i^2 b_j^2 + \frac{29}{144} \sum b_i^2 + \frac{17}{12}$ \\
1 & 3 & 2 & $\frac{1}{4608} \sum b_i^6 + \frac{1}{768} \sum b_i^4 b_j^2 + \frac{1}{384} b_1^2b_2^2b_3^2 + \frac{43}{4608} \sum b_i^4 + \frac{1}{24} \sum b_i^2 b_j^2 + \frac{277}{4608} \sum b_i^2 + \frac{1}{512} b_3^4 + \frac{1}{1536} b_3^2 + \frac{81}{256}$ \\
2 & 1 & 0 & $\frac{1}{1769472}b_1^8 + \frac{3}{40960}b_1^6 + \frac{133}{61440} b_1^4 + \frac{1087}{34560} b_1^2 + \frac{247}{1440}$ \\
0 & 6 & 0 & $\frac{1}{384} \sum b_i^6 + \frac{3}{28} \sum b_i^4 b_j^2 + \frac{3}{32} \sum b_i^2 b_j^2 b_k^2 + \frac{1}{6} \sum b_i^4 + \frac{9}{6} \sum b_i^2 b_j^2 + \frac{109}{24} b_i^2 + 34$ \\ \bottomrule
\end{tabularx}
\end{center}

\bibliographystyle{plain}
\bibliography{lattice-points-Mgnbar-TR.bib}

\end{document}